\documentclass[10pt]{article}
\usepackage{verbatim}
\usepackage{amsmath}
\usepackage{amssymb}
\usepackage{latexsym}
\usepackage{staves}
\usepackage{marvosym}
\usepackage{eurosym}
\usepackage{phaistos}
\usepackage{verbatim}
\usepackage{ulem}
\usepackage{bm}
\usepackage{amsthm}
\usepackage{color}
\usepackage[T1]{fontenc}
\usepackage[english]{babel}
\usepackage{graphicx}
\usepackage{tikz}
\usetikzlibrary{arrows}
\usetikzlibrary{arrows,decorations.markings}
\usetikzlibrary{mindmap,trees}

\newtheorem{Theorem}{Theorem}[part]
\newtheorem{Definition}{Definition}[part]
\newtheorem{Proposition}{Proposition}[part]

\newtheorem{Assumption}{Assumption}[part]

\newtheorem{Corollary}{Corollary}[part]
\newtheorem{Remark}{Remark}[part]

\makeatletter \@addtoreset{equation}{section}

\@addtoreset{Definition}{section}

\@addtoreset{Theorem}{section}

\@addtoreset{Proposition}{section}

\@addtoreset{Property}{section}

\@addtoreset{Assumption}{section}

\@addtoreset{Corollary}{section}

\@addtoreset{Lemma}{section}

\@addtoreset{Remark}{section}

\@addtoreset{Example}{section}

\def\essinf{{\rm ess}\, \inf\limits}
\def\esssup{{\rm ess}\,\sup\limits}

\def\Fc{{\cal F}}

\def\05{\frac{1}{2}}
\def\-1{^{-1}}
\def\1{{1\hspace{-1mm}{\rm I}}}
\def\={\;=\;}
\def\.{\;.}

\title{  }
\author{ }

\def\be{\begin{eqnarray}}
\def\ee{\end{eqnarray}}

\def\b*{\begin{eqnarray*}}
\def\e*{\end{eqnarray*}}

\def\E{\mathbb{E}}

\def\R{\R}

\def\P{\mathbb{P}}

\def \ind{{\bf 1}}

\def \R{I\!\!R}

\def\Fc{{\cal F}}

\def\-1{^{-1}}

\def\0.5{\frac{1}{2}}

\def\={\;=\;}
\def\.{\;.}

\def\1{{\bf 1}}

\def\Ybf{{\bf Y}}

\def \sp {\mathcal{S}^p}
\def \spn {\mathcal{S}^{p,\otimes n}}
\def \lpn {\mathbb{L}^{p,\otimes n}}

\title{Zero-sum mean-field Dynkin games:  characterization and convergence}

\author{Boualem Djehiche\thanks{Department of Mathematics, KTH Royal Institute of Technology, Stockholm, Sweden, email: \texttt{boualem@kth.se}} \and Roxana Dumitrescu \thanks{Department of Mathematics, King's College London, United Kingdom, email: \texttt{roxana.dumitrescu@kcl.ac.uk}}}

\begin{document}

\maketitle

\begin{abstract}
 We introduce a zero-sum game problem of mean-field type as an extension of the classical zero-sum Dynkin game problem to the case where the payoff processes might depend on the value of the game and its probability law. We establish sufficient conditions under which such a game admits a value and a saddle point. Furthermore, we provide a characterization of the value of the game in terms of a specific class of doubly reflected backward stochastic differential equations (BSDEs) of mean-field type, for which we derive an existence and uniqueness result. We then introduce a corresponding system of weakly interacting zero-sum Dynkin games and show its well-posedness. Finally, we provide a propagation of chaos result for the value of the zero-sum mean-field Dynkin game.
\end{abstract}




\section{Introduction}
The Dynkin game as introduced in \cite{dyn} is a two-persons game extension (or a variant) of an optimal stopping problem.  It is  extensively used in various applications including wars of attrition (see, e.g. \cite{ghem, hendricks, maynard, ekstrom2}, pre-emption games (see, e.g. \cite{fudenberg}), duels (see, e.g. \cite{bellman, blackwell, shapley} and the surveys by Radzik and Raghavan \cite{radzig} and in financial applications including game options (see, e.g. \cite{bieleckia, ekstrom1, grenadier,h06, kif00} and the survey by Kifer \cite{kif13}. 

The general setup for a zero-sum Dynkin game over a finite time interval $[0,T]$ (henceforth sometimes called Dynkin game)
consists of Player 1 choosing to stop the game at a stopping time $\tau$ and Player 2 choosing to stop it at a stopping  time $\sigma$. At $\tau\wedge \sigma:=\min(\tau,\sigma)$ the game is over and Player 1 pays Player 2 the amount
$$
\mathcal{J}(\tau,\sigma):=\chi_{\tau}\ind_{\{\tau\le \sigma<T\}}+\zeta_{\sigma}\ind_{\{\sigma< \tau \}}+\xi\ind_{\{\tau\wedge \sigma=T\}},
$$
where the payoffs $\chi, \zeta$ and  $\xi$ are given processes satisfying $\chi_t \le \zeta_t, \,0\le t<T$ and $\chi_T=\zeta_T=\xi$.
The objective of Player 1 is to choose $\tau$ from a set of admissible stopping times to minimize the expected value 
$J_{\tau,\sigma}:=\E[\mathcal{J}(\tau,\sigma)]$, while Player 2 chooses $\sigma$ from the same set of admissible stopping times to maximize it. 
In the last few decades, Dynkin games have been extensively studied under several sets of assumptions including \cite{alariot, bay, bensoussan74, bismut,  karatzas01, laraki05, laraki13,lep84, morimoto, stettner, touzi} (the list being far from complete). 
The two main questions addressed in all these papers are: (1) whether the Dynkin game is  fair (or has a value) i.e.  whether the following equality holds.
$$
\underset{\tau}{\inf}\,\underset{\sigma}{\sup}\,J_{\tau,\sigma}=\underset{\sigma}{\sup}\,\underset{\tau}{\inf}\,J_{\tau,\sigma};
$$
(2) whether the game has a saddle-point i.e. whether there exists a pair of admissible strategies (stopping times) $(\tau^*,\sigma^*)$ for which we have
$$
J_{\tau^*,\sigma}\le J_{\tau^*,\sigma*}\le J_{\tau,\sigma^*}.
$$
 By a simple change of variable, the expected payoff can be generalized to include  an instantaneous payoff process $g(t)$, so that 
 $$
 J_{\tau,\sigma}:=\E\left[\mathcal{R}_0(\tau,\sigma)\right].
 $$
 where
 $$
 \mathcal{R}_t(\tau,\sigma):=\int_t^{\tau\wedge \sigma} g(s)ds+\chi_{\tau}\ind_{\{\tau\le \sigma<T\}}+\zeta_{\sigma}\ind_{\{\sigma< \tau \}}+\xi\ind_{\{\tau\wedge \sigma=T\}}, \quad 0\le t\le T.
 $$

Cvitani{\'c} and Karatzas \cite{CK96} were first to establish a link between these Dynkin games and doubly reflected stochastic differential equations (DRBSDE) with driver $g(t)$ and obstacles $\chi$ and $\zeta$, in the Brownian case when $\chi$ and $\zeta$ are continuous processes, which turns out decisive for forthcoming work which considers more general forms of zero-sum Dynkin games including the case where the obstacles are merely right continuous with left limits, and also when the underlying filtration is generated by both the Brownian motion and an independent Poisson random measure, see \cite{hp00, hh06, h06} and the references therein. 

Motivated by problems in which the players use risk measures to evaluate their  payoffs, Dumitrescu {\it et al.} \cite{dqs16} have considered 'generalized' Dynkin games where the 'classical' expectation $\E[\,\cdot\,]$ is replaced the more general nonlinear expectation $\mathcal{E}^g(\cdot)$, induced by a BSDE with jumps and a nonlinear driver $g$.

The link between Dynkin games and DRBSDEs suggested in \cite{CK96} goes as follows. Assume that under suitable conditions on $g,\chi,\zeta$ and $\xi$ defined on the filtered probability space $(\Omega,\mathcal{F},(\mathcal{F}_t)_t,\P)$, carrying out a Brownian motion $B$, there is a unique solution $(Y,Z,K^1,K^2)$ to the following DRBSDE
\begin{align}\label{RBSDE-0}
    \begin{cases}
    (i)\quad  Y_t = \xi +\int_t^T g(s)ds + (K^1_T - K^1_t)-(K^2_T-K^2_t)-\int_t^T Z_s dB_s, \quad  t \in [0,T],\\
    (ii) \quad \zeta_t \geq Y_{t}\geq \chi_t,   \quad \,\, t \in [0,T], \\
    (iii) \quad  \int_0^T (Y_{t}-\zeta_t)dK^1_t = 0, \quad \int_0^T (Y_{t}-\chi_t)dK^2_t = 0,
    \end{cases}
\end{align}
where the processes $K^1$ and $K^2$ are continuous increasing processes, such that $K^1_0=K^2_0=0$.
Then the process $Y$ admits the following representation
\begin{equation}\label{value-intro}
Y_t=\underset{\tau\ge t}{\essinf}\,\underset{\sigma\ge t}{\esssup}\,\E[\mathcal{R}_t(\tau,\sigma)| \mathcal{F}_t]=\underset{\sigma\ge t}{\esssup}\,\underset{\tau\ge t}{\essinf}\,\E[\mathcal{R}_t(\tau,\sigma)| \mathcal{F}_t],\quad 0\le t\le T.
\end{equation}
The formula \eqref{value-intro} tells us is that the 'upper-value' $\overline{V}_t:=\underset{\tau\ge t}{\essinf}\,\underset{\sigma\ge t}{\esssup}\,\E[\mathcal{R}_t(\tau,\sigma)| \mathcal{F}_t]$ and the 'lower-value' $\underline{V}_t:=\underset{\sigma\ge t}{\esssup}\,\underset{\tau\ge t}{\essinf}\,\E[\mathcal{R}_t(\tau,\sigma)| \mathcal{F}_t]$ of the game coincide and they are equal to the first component of the solution $(Y,Z,K^1,K^2)$ to the DRBSDE \eqref{RBSDE-0}, in which case the zero-sum Dynkin game has a value and is given by 
$$
Y_0=\underset{\tau}{\inf}\,\underset{\sigma}{\sup}\,\E[\mathcal{R}_0(\tau,\sigma)]=\underset{\sigma}{\sup}\,\underset{\tau}{\inf}\,\E[\mathcal{R}_0(\tau,\sigma)].
$$

\medskip
\paragraph{Main contributions.} \textcolor{black}{In this paper, motivated by applications to life insurance (see an example below)}, we suggest a generalization of the above zero-sum Dynkin game, where we consider the case where the payoff depends on the 'values' $V$ of the game and their probability laws $\P_{V}$:
$$
g(t):=f(t,V_t,\P_{V_t}),\quad \chi_t:=h_1(t,V_t,\P_{V_t}),\quad \zeta_t:=h_2(t,V_t,\P_{V_t})
$$
so that, for every $t\in [0,T]$,
$$
\mathcal{R}_t(\tau,\sigma,V):=\int_t^{\tau\wedge \sigma} f(s,V_s,\P_{V_s})ds+h_1(\tau, V_{\tau}, \P_{V_s}|_{s=\tau})\ind_{\{\tau\le \sigma<T\}}+h_2(\sigma,V_{\sigma},\P_{V_s}|_{s=\sigma})\ind_{\{\sigma< \tau \}}+\xi\ind_{\{\tau\wedge \sigma=T\}}.
$$
The upper and lower values of the game satisfy 
\begin{equation}\label{recur-0}
\overline{V}_t:=\underset{\tau\ge t}{\essinf}\,\underset{\sigma\ge t}{\esssup}\,\E[\mathcal{R}_t(\tau,\sigma,\overline{V})| \mathcal{F}_t],\qquad \underline{V}_t:=\underset{\sigma\ge t}{\esssup}\,\underset{\tau\ge t}{\essinf}\,\E[\mathcal{R}_t(\tau,\sigma,\underline{V})| \mathcal{F}_t]. 
\end{equation}
Due to the dependence of the payoff on the probability law of the 'value', we call the game whose upper and lower values satisfy \eqref{recur-0}, a \textit{zero-sum mean-field Dynkin game}. \textcolor{black}{This new type of games are  more involved then the standard Dynkin games, since the first question one has to answer is the existence of the upper (resp. lower) value. The first main result of this paper is to show that, when the underlying filtration is generated by both the Brownian motion and an independent Poisson random measure, under mild regularity assumptions on the payoff process, the upper and lower values $\overline{V}$ and $\underline{V}$ of the game exist and are unique. Then, we show this game has  a value (i.e. $\overline{V}_t=\underline{V}_t, \,0\le t\le T,\,\, \P$-a.s.), which can be characterized in terms of the component $Y$ of the solution of \textit{a new class of mean-field doubly reflected BSDEs}, whose obstacles might depend on the solution and its distribution (see Section \ref{MFDG}). We prove the results in a general setting when the driver of the doubly reflected BSDE might also depend on $z$ and $u$ and the method we propose to show the existence of the value of the game (and the solution of the corresponding mean-field doubly reflected BSDE) is based on a new approach which avoids the standard penalization technique. We also provide sufficient conditions under which the game admits \textit{a saddle-point}, the main difficulty in our framework being due to the dependence of the obstacles on the value of the game. The second main result consists in proving the existence of the value of the following system of {\it interacting Dynkin games} which take the form
\begin{equation}\label{recur-1}
\overline{V}^{i,n}_t:=\underset{\tau\ge t}{\essinf}\,\underset{\sigma\ge t}{\esssup}\,\E[\mathcal{R}^{i,n}_t(\tau,\sigma,\overline{V}^{n})| \mathcal{F}_t],\qquad \underline{V}^{i,n}_t:=\underset{\sigma\ge t}{\esssup}\,\underset{\tau\ge t}{\essinf}\,\E[\mathcal{R}^{i,n}_t(\tau,\sigma,\underline{V}^{n})| \mathcal{F}_t], 
\end{equation}
where
$$\begin{array}{lll}
\mathcal{R}^{i,n}_t(\tau,\sigma,V^{n}):=\int_t^{\tau\wedge \sigma} f(s,V^{i,n}_s,\frac{1}{n}\sum_{j=1}^n\delta_{V^{j,n}_s})ds+h_1(\tau, V^{i,n}_{\tau}, \frac{1}{n}\sum_{j=1}^n\delta_{V^{j,n}_s}|_{s=\tau})\ind_{\{\tau\le \sigma<T\}} \\ \qquad\qquad \qquad\qquad +h_2(\sigma,V^{i,n}_{\sigma}, \frac{1}{n}\sum_{j=1}^n\delta_{V^{j,n}_s}|_{s=\sigma})\ind_{\{\sigma< \tau\}}+\xi^{i,n}\ind_{\{\tau\wedge \sigma=T\}}.
\end{array}
$$
and providing sufficient conditions under which a saddle point exists. We also show the link with the solution of a system of interacting doubly reflected BSDEs, with obstacles depending on the solution, for which the well-posedness is addressed in the general case when $f$ might also depend on $z$ and $u$. The third main contribution is a convergence result, which shows that, under appropriate assumptions on the involved coefficients, the value $V$ is limit (as $n\to \infty$), under appropriate norms, of ${V}^{i,n}$ (the value of the interacting zero-sum Dynkin game). A related propagation of chaos type result is derived.}

\subsection{Motivating example from life insurance} One of the main motivations of studying the class of zero-sum mean-field Dynkin games \eqref{recur-0} and \eqref{recur-1} is the pricing of the following prospective reserves in life insurance.
Consider a  portfolio of a large number $n$ of homogeneous life insurance policies $\ell$. Denote by $Y^{\ell,n}$ the prospective reserve of each policy $\ell=1,\ldots,n$. Life insurance is a business which reflects the cooperative aspect of the pool of insurance contracts.  To this end, the prospective reserve is constructed and priced based on {\it the averaging principal} where an individual reserve $Y^{\ell,n}$ is compared with the average reserve $\frac{1}{n}\sum_{j=1}^n Y^{j,n}$ a.k.a. {\it model point} among actuaries. In particular, in nonlinear reserving, the driver $f$ i.e. the reward per unit time, the solvency/guarantee level $h_1$ (lower barrier) and the allocated bonus level $h_2$ (upper barrier) i.e. the  fraction of the value of the global market portfolio of the insurance company allocated to the prospective reserve, depend on the reserve of the particular contract and on the average reserve characteristics over the $n$ contracts (since $n$ is very large, averaging over the remaining $n-1$ policies has roughly the same effect as averaging over all $n$ policies): For each $\ell=1,\ldots,n$,
\begin{equation}\label{MF-N}\begin{array}{lll}
f(t,Y^{\ell,n}_t,(Y^{m,n}_t)_{m\neq\ell}):=\alpha_t-\delta_tY_t+\beta_t\max(\theta_t, Y^{\ell,n}_t-\frac{1}{n}\sum_{k=1}^nY^{k,n}_t),\\
h_1(Y^{\ell,n}_t,(Y^{m,n}_t)_{m\neq\ell})=\left(u-c^1(Y^{\ell,n}_t)+\mu(\frac{1}{n}\sum_{k=1}^nY^{k,n}_t-u)^+\right) \wedge S_t, \\
h_2(Y^{\ell,n}_t,(Y^{m,n}_t)_{m\neq\ell})=\left(c^2(Y^{\ell,n}_t)+c^3(\frac{1}{n}\sum_{k=1}^nY^{k,n}_t)\right) \vee S^{\prime}_t,
\end{array}
\end{equation}
where $0<\mu<1$,  $S$ is the value of the 'benchmark' global portfolio of the company and $S^{\prime}$  some higher value of that global portfolio used by the company as a reference (threshold) to apply the bonus allocation program, where at each time $t$, $S_t\le S^{\prime}_t$ and the involved functions $c^1,c^2$ and the parameters $u,\mu$ are chosen so that $h_1(Y^{\ell,n}_t,(Y^{m,n}_t)_{m\neq\ell})\le h_2(Y^{\ell,n}_t,(Y^{m,n}_t)_{m\neq\ell})$. The driver $f$ includes the discount rate $\delta_t$ and  deterministic positive functions   
$\alpha_t, \beta_t$ and $\theta_t$ which constitute the elements of the withdrawal option. The solvency level $h_1$ is constituted of a required minimum of a benchmark return (guarantee) $u$, a reserve dependent management fee  $c^1(Y^{\ell,n}_t)$ (usually much smaller than $u$) and  a 'bonus' option $(\frac{1}{n}\sum_{k=1}^nY^{k,n}_t-u)^+$ which is the possible surplus realized by the average of all involved contracts. The allocated bonus level $h_2$ is usually prescribed by the contract and includes a function $c^2(Y^{\ell,n}_t)$ which reflects a possible bonus scheme based the individual reserve level and  another function $c^3(\frac{1}{n}\sum_{k=1}^nY^{k,n}_t)$ which reflects the average reserve level.

The Dynkin game is between two players, the insurer, Player (I), and each of the $N$ insured (holders of the insurance contracts), Player ($H_{\ell}$), where each of them can decide to stop it i.e. exit the contract at a random time of her choice. Player (I) stops the game when the solvency level $h_1(Y^{\ell,n}_t,(Y^{m,n}_t)_{m\neq\ell})$ 
of the $\ell$th player is reached, while Player ($H_{\ell}$) stops the game when its allocated bonus level $h_2(Y^{\ell,n}_t,(Y^{m,n}_t)_{m\neq\ell})$ is reached.

The prospective reserve $Y^{\ell,n}$ is an upper value (resp. lower value) for the game if Player ($H_{\ell}$) (resp. Player (I)) acts first and then Player (I) (resp. Player ($H_{\ell}$)) chooses an optimal response.

\medskip Sending $n$ to infinity in \eqref{MF-N}, yields the following forms of {\it upper or lower value dependent} payoffs of the prospective reserve of a representative (model-point) life insurance contract:
\begin{equation*}\begin{array}{lll}
f(t,Y_t,\E[Y_t]):=\alpha_t-\delta_tY_t+\beta_t\max(\theta_t, Y_t- \E[Y_t]),\\
h_1(Y_t, \E[Y_t])=\left(u^2-c^1(Y_t)+\mu^1(\E[Y_t]-u)^+\right)\wedge S_t, \\ 
h_2(Y_t, \E[Y_t])=\left(c^2(Y_t)+c^3(\E[Y_t])\right)\vee S^{\prime}_t.
\end{array}
\end{equation*}

\subsection{Organization of the paper} In Section \ref{MFDG} we introduce \textit{a class of zero-sum mean-field Dynkin games}. Under mild regularity assumptions on the coefficients involved in the payoff function, we show existence and uniqueness of the upper and lower values of the game. We also show that the game has a value and characterize it as the unique solution to a mean-field doubly reflected BSDE. Moreover, we give a sufficient condition on the obstacles which guarantees existence of a saddle-point. In Section \ref{WIDG}, we introduce a system of $n$ interacting zero-sum Dynkin games and show that it has a value and characterize it as the unique solution to a system of interacting doubly reflected BSDEs. Furthermore, we give sufficient conditions on the barriers which guarantee existence of a saddle-point for each component of the system. Finally, in Section \ref{chaos}, we show that the limit, as $n\to \infty$, of the value of system of interacting zero-sum Dynkin games converges to the value of the zero-sum mean-field Dynkin game in an appropriate norm. As a consequence of that limit theorem,  we establish a propagation of chaos property for the value of the system of interacting zero-sum Dynkin games.

\subsection*{Notation.}
\quad Let $(\Omega, \mathcal{F},\mathbb{P})$ be a complete probability space. $B=(B_t)_{0\leq t\leq T}$  is a standard $d$-dimensional Brownian motion and $N(dt,de)$ is a Poisson random measure, independent of $B$, with compensator $\nu(de)dt$ such that $\nu$ is a $\sigma$-finite measure on $\R^*$, equipped with its Borel field $\mathcal{B}(\R^*)$. Let $\tilde{N}(dt,du)$ be its compensated process. We denote by $\mathbb{F} = \{\mathcal{F}_t\}$ the natural filtration associated with $B$ and $N$. Let $\mathcal{P}$ be the $\sigma$-algebra on $\Omega \times [0,T]$ of $\mathcal{F}_t$-progressively measurable sets. Next, we introduce the following spaces with $p>1$:
\begin{itemize}
    \item $\mathcal{T}_t$ is the set of $\mathbb{F}$-stopping times $\tau$ such that
    $\tau \in [t,T]$ a.s.
    \item \textcolor{black}{$L^p(\mathcal{F}_T)$ is the set of random variables $\xi$ which are $\mathcal{F}_T$-measurable and $\mathbb{E}[|\xi|^p]<\infty$.}
    \item $\mathcal{S}_{\beta}^p$ is the set of real-valued c{\`ad}l{\`a}g adapted processes $y$ such that $||y||^p_{\mathcal{S}_{\beta}^p} :=\mathbb{E}[\underset{0\leq u\leq T}{\sup} e^{\beta p s}|y_u|^p]<\infty$. We set $\mathcal{S}^p=\mathcal{S}_{0}^p$.
    
    \item $\mathcal{S}_{\beta,i}^p$ is the subset of $\mathcal{S}_{\beta}^p$ such that the process $k$ is non-decreasing and $k_0 = 0$. We set $\mathcal{S}_{i}^p=\mathcal{S}_{0,i}^p$.
    
    \item $\mathbb{L}_{\beta}^p$ is the set of real-valued c{\`ad}l{\`a}g adapted processes $y$ such that $||y||^p_{\mathbb{L}_{\beta}^p} :=\underset{\tau\in\mathcal{T}_0}{\sup}\E[e^{\beta p \tau}|y_{\tau}|^p]<\infty$. $\mathbb{L}_{\beta}^p$ is a Banach space (see Theorem 22 in \cite{DM82}, pp. 60 when $p=1$). We set $\mathbb{L}^p=\mathbb{L}_{0}^p$. 

    \item $\mathcal{H}^{p,d}$ is the set of $\mathcal{P}$-measurable, $\R^d$-valued processes such that  $\mathbb{E}[(\int_0^T|v_s|^2ds)^{p/2}] <\infty$.
    
     \item $L^p_{\nu}$ is the set of measurable functions $l:\R^*\to \R$ such that $\int_{\R^*}|l(u)|^p\nu(du)<+\infty$.
    The set $L^2_{\nu}$ is a Hilbert space equipped with the scalar product $\langle\delta,l\rangle_{\nu}:=\int_{\R^*}\delta(u)l(u)\nu(du)$ for all $(\delta,l)\in L^2_{\nu}\times L^2_{\nu}$, and the norm $|l|_{\nu,2}:=\left(\int_{R^*}|l(u)|^2\nu(du)\right)^{1/2}$. If there is no risk for confusion, we sometimes denote $|l|_{\nu,2}:=|l|_{\nu}$.
    
    \item $\mathcal{B}(\R^d)$ (resp. $\mathcal{B}(L^p_{\nu}))$ is the Borel $\sigma$-algebra on $\R^d$ (resp. on $L^p_{\nu}$).
    
    \item $\mathcal{H}^{p,d}_{\nu}$ is the set of predictable processes $l$, i.e. measurable 
    \[l:([0,T]\times \Omega \times \R^*,\mathcal{P}\otimes \mathcal{B}(\R^*))\to (\R^d, \mathcal{B}(\R^d)); \quad (\omega,t,u)\mapsto l_t(\omega,u)\]
    such that $\|l\|_{\mathcal{H}^{p,d}_{\nu}}^p:=\mathbb{E}\left[\left(\int_0^T\sum_{j=1}^d|l^j_t|^2_{\nu}dt\right)^{\frac{p}{2}}\right]<\infty$.  For $d=1$, we denote  $\mathcal{H}^{p,1}_{\nu}:=\mathcal{H}^{p}_{\nu}$.
    
    \item $\mathcal{P}_p(\R)$ is the set of probability measures on $\R$ with finite $p$th moment. We equip the space $\mathcal{P}_p(\R)$ with the $p$-Wasserstein distance denoted by $\mathcal{W}_p$ and defined as
    \begin{align*}
        \mathcal{W}_p(\mu, \nu) := \inf \left\{\int_{\R \times \R} |x - y|^p \pi(dx, dy) \right\}^{1/p},
    \end{align*}
    where the infimum is over probability measures $\pi \in\mathcal{P}_p(\R\times \R)$ with first and second marginals $\mu$ and $\nu$, respectively.

\end{itemize}

\section{Zero-sum mean-field Dynkin games and link with mean-field doubly reflected BSDEs}\label{MFDG}

\noindent In this section, we introduce a new class of zero-sum mean-field Dynkin games which have the particularity that the payoff depends on the value of the game, which is shown to exist under specific assumptions. For this purpose, we first recall the notion of $f$-conditional expectation introduced by Peng (see e.g. \cite{Peng07}), which is denoted by $\mathcal{E}^{f}$ and extends the standard conditional expectation to the nonlinear case.

\begin{Definition}[Conditional $f$-expectation] We recall that if $f$ is a Lipschitz driver and $\xi$ a random variable belonging to $L^p(\mathcal{F}_T)$, then there exists a unique solution $(X,\pi, \psi) \in \mathcal{S}^p \times \mathcal{H}^{p,d} \times \mathcal{H}_\nu^{p}$ to the following BSDE
\begin{align*}
X_t=\xi+\int_t^T f(s,X_s,\pi_s,\psi_s)ds-\int_t^T\pi_sdW_s -\int_t^T \int_{\R^\star}\pi_s(de) d\tilde{N}(ds,de)\text{ for all } t \in [0,T] {\rm \,\, a.s. }
\end{align*}
For $t\in [0,T]$, the nonlinear operator $\mathcal{E}^{f}_{t,T} : L^p(\mathcal{F}_T) \mapsto L^p(\mathcal{F}_t)$ which maps a given terminal condition $\xi \in L^p(\mathcal{F}_T)$ to the first component $X_t$ at time $t$ of the solution of the above BSDE is called conditional $f$-expectation at time $t$. It is also well-known that this notion can be extended to the case where the (deterministic) terminal time $T$ is replaced by a general stopping time $\tau \in \mathcal{T}_0$ and $t$ is replaced by a stopping time $S$ such that $S \leq \tau$ a.s.
\end{Definition}

\noindent In the sequel, given a Lipschitz continuous driver $f$ and a process $Y$, we denote by $f \circ Y$ the map $(f \circ Y)(t,\omega,y,z,u):= f(t,\omega,y,z,u, \mathbb{P}_{Y_t})$.\\

\noindent We introduce the following definitions.

\begin{Definition} Consider the map $\underline{\Psi}:\mathbb{L}_\beta^p \longrightarrow \mathbb{L}_\beta^p$ given by $$\underline{\Psi}(Y)_t:=\underset{\tau \in \mathcal{T}_t}{\esssup}\,\underset{\sigma \in \mathcal{T}_t}{\essinf}\,\mathcal{E}_{t,\tau\wedge \sigma}^{f \circ Y}[h_1(\tau, Y_\tau, \mathbb{P}_{Y_s}|_{s=\tau})\ind_{\{\tau \leq \sigma<T\}}+h_2(\sigma, Y_\sigma, \mathbb{P}_{Y_s}|_{s=\sigma})\ind_{\{\sigma < \tau\}}+\xi\ind_{\{\sigma \wedge \tau = T\}}].$$ 
We define the \textit{first} or \textit{lower value function} of the zero-sum mean-field Dynkin game, denoted by $\underline{V}$, as the \textit{fixed point} of the application $\underline{\Psi}$,
i.e. it satisfies 
\begin{align}\label{lower}
    \underline{V}_t=\underline{\Psi}(\underline{V})_t.
\end{align}
\end{Definition}
\begin{Definition} Let the map $\overline{\Psi}:\mathbb{L}_\beta^p \longrightarrow \mathbb{L}_\beta^p$ be given by
$$\overline{\Psi}(Y)_t:=\underset{\sigma \in \mathcal{T}_t}{\essinf}\, \underset{\tau \in \mathcal{T}_t}{\esssup}\,\mathcal{E}_{t,\tau\wedge \sigma}^{f \circ Y}[h_1(\tau, Y_\tau, \mathbb{P}_{Y_s}|_{s=\tau})\ind_{\{\tau \leq \sigma<T\}}+h_2(\sigma, Y_\sigma, \mathbb{P}_{Y_s}|_{s=\sigma})\ind_{\{\sigma < \tau\}}+\xi\ind_{\{\sigma \wedge \tau = T\}}].$$
The \textit{second} or \textit{upper value function} of the zero-sum mean-field Dynkin game, denoted by $\overline{V}$, as the \textit{fixed point} of the application $\overline{\Psi}$, i.e. it satisfies 
\begin{align}\label{upper}
    \overline{V}_t=\overline{\Psi}(\overline{V})_t.
\end{align}
\end{Definition}
\begin{Definition} The zero-sum mean-field Dynkin game is said to admit a \textit{common value function}, called the \textit{value of the game},
if $\underline{V}$ and $\overline{V}$ exist and $\underline{V}_t=\overline{V}_t$ for all $t \in [0,T]$, $\P$-a.s.
\end{Definition}
More precisely, the value of the game, denoted by $V$, corresponds to the common fixed point of the applications $\overline{\Psi}$ and $\underline{\Psi}$, i.e. it satisfies $V_t=\overline{\Psi}(V)_t=\underline{\Psi}(V)_t$.
The main result of this section consists in providing conditions under which the game admits a value, showing the existence and uniqueness of the solution of a \textit{new class} of mean-field doubly reflected BSDEs given below and establishing the connection between the value of the mean-field Dynkin game and the solution of the mean-field  reflected BSDE.\\

\noindent Let us introduce the new class of doubly reflected BSDEs of mean-field type associated with the driver $f$, the terminal condition $\xi$, lower barrier $h_1$ and upper barrier $h_2$.

\begin{Definition}
We say that the quadruple of progressively measurable processes $(Y_t, Z_t, U_t, K^2_t, K_t^2)_{t\leq T}$ is a solution of the mean-field reflected BSDE associated with $(f,\xi, h_1, h_2)$ if, when $p\geq 2$,
\begin{align}\label{BSDE1}
    \begin{cases}
    (i) \quad Y \in \mathcal{S}^p,  Z\in \mathcal{H}^{p,d}, U\in \mathcal{H}^{p}_{\nu}, K^1 \in \mathcal{S}_i^p \text{ and }  K^2 \in \mathcal{S}_i^p\\
    (ii)\quad  Y_t = \xi +\int_t^T f(s,Y_{s}, Z_{s}, U_{s}, \P_{Y_s})ds + (K^1_T - K^1_t)-(K^2_T-K^2_t) \\ \qquad\qquad -\int_t^T Z_s dB_s - \int_t^T \int_{\R^\star} U_s(e) \tilde N(ds,de), \quad  t \in [0,T],\\
    (iii) \quad h_2(t,Y_{t},\P_{Y_t}) \geq Y_{t}\geq h_1(t,Y_{t},\P_{Y_t}), \quad \forall t \in [0,T],\\
    (iv) \quad  \int_0^T (Y_{t-}-h_1(t,Y_{t-},\P_{Y_{t-}}))dK^1_t = 0, \,\,\,\int_0^T (Y_{t-}-h_2(t,Y_{t-},\P_{Y_{t-}}))dK^2_t = 0, \\ (v) \quad dK^1_t \perp dK^2_t,\quad t \in [0,T]. \quad 
    \end{cases}
\end{align}
\end{Definition}
The last condition $dK^1_t \perp dK^2_t$ is imposed in order to ensure the uniqueness of the solution. For the reader's convenience, we recall here the definition of a mutually singular measures associated with increasing predictable processes.
\begin{Definition}
Let $A=(A_t)_{0 \leq t \leq T}$ and $A'=(A'_t)_{0 \leq t \leq T}$ belonging to $\mathcal{S}_{i}^p$. The measures $dA_t$ and $dA'_t$ are said to be mutually singular and we write $dA_t \perp dA'_t$ if there exists $D \in \mathcal{P}$ such that 
$$\mathbb{E}\left[\int_0^T 1_{D^c}dA_t\right]=\mathbb{E}\left[\int_0^T 1_{D}dA'_t\right]=0.$$
\end{Definition}
We make the following assumption on $(f,h_1,h_2,\xi)$.
\begin{Assumption} \label{generalAssump}
The coefficients $f,h_1$, $h_2$ and $\xi$ satisfy the following properties.
\begin{itemize} 
    \item[(i)] $f$ is a mapping from $[0,T] \times \Omega \times \R \times \R  \times L^p_{\nu} \times \mathcal{P}_p(\R)$ into $\R$ such that
    \begin{itemize}
        \item[(a)] the process $(f(t,0,0,0,\delta_0))_{t\leq T}$ is $\mathcal{P}$-measurable and belongs to $\mathcal{H}^{p,1}$;
        \item[(b)] $f$ is Lipschitz continuous w.r.t. $(y,z,u,\mu)$ uniformly in $(t,\omega)$, i.e. there exists a positive constant $C_f$ such that $\mathbb{P}$-a.s. for all $t\in [0,T],$ 
        \[|f(t,y_1,z_1,u_1,\mu_1)-f(t,y_2,z_2,u_2,\mu_2)|\leq C_f(|y_1-y_2|+|z_1-z_2|+|u_1-u_2|_{\nu}+\mathcal{W}_p(\mu_1, \mu_2))\]
        for any $y_1,y_2\in \R,$ $z_1, z_2 \in \R$, $u_1, u_2 \in L^p_{\nu}$ and $\mu_1, \mu_2 \in \mathcal{P}_p(\R) $.
        \item[(c)] Assume that $d\mathbb{P} \otimes dt$ a.e. for each $(y,z,u_1,u_2,\mu) \in \R^2 \times (L^2_\nu)^2 \times \mathcal{P}_p(\R),$
        \begin{align}
            f(t,y,z,u_1,\mu)-f(t,y,z,u_2,\mu) \geq \langle\gamma_t^{y,z,u_1,u_2,\mu},l_1-l_2\rangle_\nu,
        \end{align}
        with 
        \begin{align*}
        \gamma:[0,T] \times \Omega \times \R^2 \times (L^2_\nu)^2 \times \mathcal{P}_p(\R) \mapsto L^2_\nu;\\
        (\omega,t,y,z,u_1,u_2,\mu) \mapsto \gamma_t^{y,z,u_1,u_2,\mu}(\omega,\cdot)
        \end{align*}
        $\mathcal{P} \otimes \mathcal{B}(\R^2) \otimes \mathcal{B}((L^2_\nu)^2) \otimes \mathcal{B}(\mathcal{P}_p(\R))$ measurable satisfying $\|\gamma_t^{y,z,u_1,u_2,\mu}(\cdot)\|_\nu \leq C$ for all\\ $(y,z,u_1,u_2,\mu) \in \R^2 \times (L^2_\nu)^2 \times \mathcal{P}_p(\R)$, $d\mathbb{P} \otimes dt$-a.e., where $C$ is a positive constant, and such that $\gamma_t^{y,z,u_1,u_2,\mu}(e) \geq -1$, for all $(y,z,u_1,u_2,\mu) \in \R^2 \times (L^2_\nu)^2 \times \mathcal{P}_p(\R)$, $d\mathbb{P} \otimes dt \otimes d\nu(e)$-a.e.
        \end{itemize}
    \item[(ii)] $h_1$, $h_2$ are measurable mappings from $[0,T] \times \Omega \times \R \times \mathcal{P}_p(\R)$ into $\R$ such that $h_1(t,y,\mu):=\tilde{h}_1(t,y,\mu) \wedge S_t$ and $h_2(t,y,\mu):=\tilde{h}_2(t,y,\mu) \vee S'_t$,  where
    \begin{itemize}
        \item[(a)] $S_t$ and $S'_t$ are quasimartingales  in $S^p$, with $S_t \leq S'_t$ $\mathbb{P}$-a.s. for all $0 \leq t \leq T$,
        \item[(b)] the processes $\left(\sup_{(y,\mu)\in \R\times\mathcal{P}_p(\R)}|h_1(t,y,\mu)|\right) _{0\le t\leq T}$ and $\left(\sup_{(y,\mu)\in \R\times\mathcal{P}_p(\R)}|h_2(t,y,\mu)|\right) _{0\le t\leq T}$ belong to $\mathcal{S}^p$,
        \item[(c)] $\tilde{h}_1$ (resp. $\tilde{h}_2$) is Lipschitz w.r.t. $(y,\mu)$ uniformly in $(t,\omega)$, i.e. there exists two positive constants $\gamma_1$ (resp. $\kappa_1$ ) and $\gamma_2$ (resp. $\kappa_2$ ) such that $\mathbb{P}$-a.s. for all $t\in [0,T],$ 
        \[|\tilde{h}_1(t,y_1,\mu_1)-\tilde{h}_1(t,y_2,\mu_2)|\leq \gamma_1|y_1-y_2|+\gamma_2\mathcal{W}_p(\mu_1, \mu_2),\]
         \[|\tilde{h}_2(t,y_1,\mu_1)-\tilde{h}_2(t,y_2,\mu_2)|\leq \kappa_1|y_1-y_2|+\kappa_2\mathcal{W}_p(\mu_1, \mu_2)\] for any $y_1,y_2\in \R$ and $\mu_1,\mu_2 \in \mathcal{P}_p(\R)$, 
        \item[(d)] $h_1(t,y,\mu) \leq h_2(t,y,\mu)$ for all $y \in \R$ and $\mu \in \mathcal{P}_p(\R)$. 
    \end{itemize}
    \item[(iii)] $\xi \in L^p(\Fc_T)$ and satisfies $h_2(T,\xi,\P_{\xi}) \geq \xi \geq h_1(T,\xi,\P_{\xi})$.
    \end{itemize}
\end{Assumption}

\begin{Remark}\label{moko}
The above assumptions (ii) on the obstacles $h_1$ and $h_2$ imply the Mokobozki condition 
$$
h_1(t,y,\mu)\le  S^{\prime}_t\le h_2(t,y,\mu), \quad \text{for all} \,\,\, (t,y,\mu)\in [0,T]\times \R \times \mathcal{P}_p(\R), \quad \P\text{-a.s.},
$$
since $S^{\prime}$ is a quasimartingale.
\end{Remark}
\begin{Remark}
We note that if $\tilde{h}_1$ (resp. $\tilde{h}_2$) depends only on $\mu$ i.e. $\tilde{h}_1(t,y,\mu)=\tilde{h}_1(t,\mu)$ (resp.  $\tilde{h}_2(t,y,\mu)=\tilde{h}_2(t,\mu)$), the domination condition (ii)(b) can be dropped.
\end{Remark}

\begin{Theorem}[\textit{Existence of the value and link with mean-field doubly reflected BSDEs}]\label{Dynkin}
Suppose that Assumption \ref{generalAssump} is in force for some $p\geq 2$. Assume that  $\gamma_1$, $\gamma_2$, $\kappa_1$ and $\kappa_2$ satisfy 
\begin{align} \label{ExistenceCond}
    \gamma_1^p+\gamma_2^p+\kappa_1^p+ \kappa_2^p < 2^{3-\frac{5p}{2}}. 
\end{align}
Then,
\begin{itemize}
\item[(i)]The \textit{mean-field Dynkin game} admits a value $V\in \mathcal{S}^{p}$.  

\item[(ii)] The mean-field doubly reflected BSDEs \eqref{BSDE1} has a unique solution $(Y,Z,U,K^{1},K^{2})$ in $\mathcal{S}^{p} \otimes \mathcal{H}^{p} \otimes \mathcal{H}_{\nu}^{p} \otimes \mathcal{S}_{i}^{p} \otimes \mathcal{S}_{i}^p$.
\item[(iii)] We have $V_\cdot=Y_\cdot$.
\end{itemize}
\end{Theorem}

\begin{Remark}
In the Brownian motion case, existence and uniqueness of the solution to \eqref{BSDE1} is derived in  \cite{CHM20}, in the particular case when the mean-field coupling is in terms of $\E[Y_t]$ and the driver $f$ does not depend on $z$ and $u$, under the so-called 'strict separation of obstacles' condition. The result in \cite{CHM20} is obtained  under a smallness condition different from \eqref{ExistenceCond}. 
\end{Remark}

\begin{proof} 
  \noindent \uline{\textit{Step 1}} (\textit{Well-posedness and contraction property of the operators $\overline{\Psi}$ and $\underline{\Psi}$}).  We derive these properties only for the operator $\overline{\Psi}$ since a similar proof holds for the operator $\underline{\Psi}$.
  
  \medskip 
  
  We first show that $\overline{\Psi}$  is a well-defined map from $\mathbb{L}^{p}_{\beta}$ to itself. Indeed, let $\bar{Y} \in \mathbb{L}^{p}_{\beta}$. Since $h_1$ and $h_2$ satisfy Assumption \ref{generalAssump} (ii), it follows that $h_1(t,\bar{Y}_t,\mathbb{P}_{\bar{Y}_t}) \in \mathcal{S}^p$ and $h_2(t,\bar{Y}_t,\mathbb{P}_{\bar{Y}_t}) \in \mathcal{S}^p$. Therefore, there exists a unique solution $(\hat{Y}, \hat{Z}, \hat{U}, \hat{K}^1, \hat{K}^2) \in \mathcal{S}^p \times \mathcal{H}^{p,1} \times \mathcal{H}^p_\nu \times (\mathcal{S}_{i}^p)^2$ of the reflected BSDE associated with the obstacle processes $h_1(t,\bar{Y}_t,\mathbb{P}_{\bar{Y}_t})$ and $h_2(t,\bar{Y}_t,\mathbb{P}_{\bar{Y}_t})$, the terminal condition $\xi$ and the driver $(f \circ {\bar{Y}})(t,\omega,y,z,u)$. Since $f$ satisfies Assumption \ref{generalAssump} {\it (i.c)}, by classical results on the link between the $Y$-component of the solution of a doubly reflected BSDE and optimal stopping games with nonlinear expectations (see e.g. \cite{dqs16}), we obtain $\overline{\Psi}(\bar{Y})=\hat{Y} \in \mathcal{S}^p \subset \mathbb{L}^{p}_{\beta}$.
  
  Let us now show that 
 $\overline{\Psi}: \mathbb{L}_{\beta}^p \longrightarrow \mathbb{L}_{\beta}^p$ is a contraction on the time interval $[T-\delta,T]$, for some small $\delta >0$ to be chosen appropriately.
\medskip
First, note that by the Lipschitz continuity of $f$ and $h$, for $Y,\bar{Y} \in \mathcal{S}^{p}_\beta$, $Z,\bar{Z} \in \mathcal{H}^{p,1}$, and $U,\bar{U} \in \mathcal{H}_\nu^{p}$,
 \begin{equation}
\begin{array}{lll}\label{f-h-lip-3}
|f(s,Y_s,Z_s,U_s, \P_{Y_s})-f(s,\bar{Y}_s,\bar{Z}_s,\bar{U}_s,\P_{\bar{Y}_s})|
\le  C_f(|Y_s-\bar{Y}_s|+|Z_s-\bar{Z}_s|+|U_s-\bar{U}_s|_{\nu}+\mathcal{W}_p(\P_{Y_s},\P_{\bar{Y}_s})), \\ \\
|h_1(s,Y_s,\P_{Y_s})-h_1(s,\bar{Y}_s,\P_{\bar{Y}_s})|\le \gamma_1|Y_s-\bar{Y}_s|+\gamma_2 \mathcal{W}_p(\P_{Y_s},\P_{\bar{Y}_s}), \\ \\
|h_2(s,Y_s,\P_{Y_s})-h_2(s,\bar{Y}_s,\P_{\bar{Y}_s})|\le \kappa_1|Y_s-\bar{Y}_s|+\kappa_2 \mathcal{W}_p(\P_{Y_s},\P_{\bar{Y}_s}).
\end{array}
\end{equation}
For the $p$-Wasserstein distance, we have the following inequality: for $0\leq s \leq u \leq t \leq T$,
\be \label{Wass-property}
\sup_{u\in [s,t]} \mathcal{W}_p(\P_{Y_u},\P_{\bar{Y}_u}) \leq \sup_{u\in [s,t]}(\mathbb{E}[|Y_u-\bar{Y}_u|^p])^{1/p},
\ee
from which we derive the following useful inequality
\be
\sup_{u\in [s,t]} \mathcal{W}_p(\P_{Y_u},\delta_0) \leq \sup_{u\in [s,t]}(\mathbb{E}[|Y_u|^p])^{1/p}.
\ee
Fix $Y,\bar{Y} \in \mathbb{L}_\beta^p$. For any $T-\delta \leq t \leq T$, by the estimates (A.1) on BSDEs (see the appendix, below), we have  
\begin{equation*}\begin{array} {lll}
|\overline{\Psi}(Y)_t-\overline{\Psi}(\bar{Y})_t|^{p} \\ \quad
=|\underset{\sigma\in\mathcal{T}_t}{\essinf\,} \underset{\tau\in\mathcal{T}_t}{\esssup\,}\mathcal{E}_{t,\tau \wedge \sigma}^{f \circ Y }[h_1(\tau,Y_{\tau},\P_{Y_s|s=\tau})\ind_{\{\tau \leq  \sigma<T\}}+h_2(\sigma,Y_{\sigma},\P_{Y_s|s=\sigma})\ind_{\{\sigma<\tau\}}+\xi \ind_{\{\tau \wedge \sigma=T\}}] \\ \quad
 -\underset{\sigma\in\mathcal{T}_t}{\essinf\,} \underset{\tau\in\mathcal{T}_t}{\esssup\,}\mathcal{E}_{t,\tau \wedge \sigma}^{f \circ \bar{Y} }[h_1(\tau,\bar{Y}_{\tau},\P_{\bar{Y}_s|s=\tau})\ind_{\{\tau \leq  \sigma<T\}}+h_2(\sigma,\bar{Y}_{\sigma},\P_{\bar{Y}_s|s=\sigma})\ind_{\{\sigma<\tau\}}+\xi \ind_{\{\tau \wedge \sigma=T\}}]|^{p}
\\ \quad \le \underset{\tau\in\mathcal{T}_t}{\esssup\,} \underset{\sigma\in\mathcal{T}_t}{\esssup\,}\left|\mathcal{E}_{t,\tau \wedge \sigma}^{f \circ Y}[h_1(\tau,Y_{\tau},\P_{Y_s|s=\tau})\ind_{\{\tau \leq  \sigma<T\}}+h_2(\sigma,Y_{\sigma},\P_{Y_s|s=\sigma})\ind_{\{\sigma<\tau\}}+\xi \ind_{\{\tau \wedge \sigma=T\}}] \right. \\ \left. \qquad - \mathcal{E}_{t,\tau \wedge \sigma}^{f \circ \bar{Y} }[h_1(\tau,\bar{Y}_{\tau},\P_{\bar{Y}_s|s=\tau})\ind_{\{\tau \leq  \sigma<T\}}+h_2(\sigma,\bar{Y}_{\sigma},\P_{\bar{Y}_s|s=\sigma})\ind_{\{\sigma<\tau\}}+\xi \ind_{\{\tau \wedge \sigma=T\}}] \right|^{p}\\ \quad
\le \underset{\tau\in\mathcal{T}_t}{\esssup\,} \underset{\sigma\in\mathcal{T}_t}{\esssup\,} \eta^p 2^{\frac{p}{2}-1}\E\left[\left(\int_t^{\tau \wedge \sigma} e^{2 \beta  (s-t)} |(f \circ Y) (s,\widehat{Y}^{\tau, \sigma}_s,\widehat{Z}^{\tau, \sigma}_s,\widehat{U}^{\tau, \sigma}_s) \right.\right. \\ \left.\left. \qquad\qquad\qquad -(f \circ \bar{Y}) (s,\widehat{Y}^{\tau, \sigma}_s, \widehat{Z}^{\tau, \sigma}_s,\widehat{U}^{\tau, \sigma}_s)|^2 ds\right)^{p/2} \right. \\ \left. \qquad\qquad\qquad\qquad +2^{p/2-1} e^{p \beta(\tau \wedge \sigma-t)}|\left(h_1(\tau,Y_\tau,\P_{Y_s|s=\tau})-h_1(\tau,\bar{Y}_\tau, \P_{\bar{Y}_s|s=\tau})\right)\ind_{\{\tau \leq  \sigma<T\}} \right. \\ \left. \qquad\qquad\qquad\qquad\qquad +\left(h_2(\sigma,Y_\sigma,\P_{Y_s|s=\sigma})-h_2(\sigma,\bar{Y}_\sigma, \P_{\bar{Y}_s|s=\sigma})\right)\ind_{\{\sigma < \tau\}}|^p
|\mathcal{F}_t\right] \\\quad  = \underset{\tau\in\mathcal{T}_t}{\esssup\,} \underset{\sigma\in\mathcal{T}_t}{\esssup\,} \eta^p 2^{\frac{p}{2}-1} \E\left[\left(\int_t^{\tau \wedge \sigma} e^{2 \beta (s-t)}|f(s,\widehat{Y}^{\tau, \sigma}_s,\widehat{Z}^{\tau, \sigma}_s,\widehat{U}^{\tau, \sigma}_s,
\P_{Y_s}) \right.\right. \\ \left.\left. \qquad\qquad\qquad\qquad\qquad\qquad\qquad -f(s,\widehat{Y}^{\tau, \sigma}_s,\widehat{Z}^{\tau, \sigma}_s,\widehat{U}^{\tau, \sigma}_s, \P_{\bar{Y}_s})|^2 ds \right)^{p/2}  \right. \\  \left. \qquad\qquad\qquad\qquad\qquad\qquad \qquad+e^{p \beta(\tau \wedge \sigma-t)}(|h_1(\tau,Y_\tau,\P_{Y_s|s=\tau})-h_1(\tau,\bar{Y}_\tau, \P_{\bar{Y}_s|s=\tau})|\right. \\ \left. \qquad\qquad \qquad\qquad\qquad\qquad\qquad\qquad\qquad\qquad+|h_2(\tau,Y_\tau,\P_{Y_s|s=\tau})-h_2(\sigma,\bar{Y}_\sigma, \P_{\bar{Y}_s|s=\sigma})|)^p
|\mathcal{F}_t\right],
 \end{array}
\end{equation*}
with $\eta$, $\beta>0$ such that  $\eta \leq \frac{1}{C_f^2}$ and  $\beta \geq 2 C_f+\frac{3}{\eta}$, where $(\widehat{Y}^{\tau, \sigma},\widehat{Z}^{\tau, \sigma},\widehat{U}^{\tau, \sigma})$ is the solution of the BSDE associated with driver $f\circ \bar{Y}$, terminal time $\tau \wedge \sigma$, terminal condition $\xi$ and terminal condition $h_1(\tau,{\bar{Y}}_\tau,\P_{\bar{Y}_s|s=\tau})\ind_{\{\tau \leq \sigma<T\}}+h_2(\sigma,{\bar{Y}}_\sigma,\P_{\bar{Y}_s|s=\sigma})\ind_{\{\sigma<\tau\}}+\xi\ind_{\{\tau \wedge \sigma=T\}}$.
Therefore, using \eqref{f-h-lip-3} and the fact that, for $\rho=\tau,\sigma$, 
\begin{equation} \label{Wass-property-1}\mathcal{W}_p^{p} (\P_{Y_s|s=\rho},\P_{\bar{Y}_s|s=\rho})\le \E[|Y_s-\bar{Y}_s|^p]_{|s=\rho}\le \underset{\rho\in \mathcal{T}_t}{\sup}\E[|Y_{\rho}-\bar{Y}_{\rho}|^p],
\end{equation}
we have, for any $t\in [T-\delta,T]$,
\begin{equation*}\label{ineq-evsn-3}
\begin{array} {lll}
e^{p \beta t}|\overline{\Psi}(Y)_t-\overline{\Psi}(\bar{Y})_t|^p \le \underset{\tau\in\mathcal{T}_t}{\esssup\,} \underset{\sigma \in\mathcal{T}_t}{\esssup\,} \E\left[\int_t^{\tau \wedge \sigma} e^{p \beta (s-t)}\delta^{\frac{p-2}{2}}2^{\frac{p}{2}-1} \eta^pC_f^p\E[|Y_s-\bar{Y}_s|^p]ds \right. \\ \left. +2^{\frac{p}{2}-1} e^{p \beta(\tau \wedge  \sigma-t)}\left\{\gamma_1|Y_\tau-\bar{Y}_\tau|+\gamma_2(\E[|Y_s-\bar{Y}_s|^p]_{|s=\tau}+\kappa_1|Y_\sigma-\bar{Y}_\sigma|+\kappa_2(\E[|Y_s-\bar{Y}_s|^p]_{|s=\sigma}\right\}^{1/p})^p
|\mathcal{F}_t\right]
\\ \qquad
 \le \underset{\tau\in\mathcal{T}_t}{\esssup\,} \underset{\sigma \in\mathcal{T}_t}{\esssup\,} \E\left[\int_t^{\tau \wedge \sigma} e^{p \beta s}2^{\frac{p}{2}-1} \delta^{\frac{p-2}{2}} \eta^pC_f^p\E[|Y_s-\bar{Y}_s|^p]ds \right. \\ \left. \qquad\qquad\qquad\qquad +2^{\frac{p}{2}-1} e^{p \beta\tau \wedge \sigma}\left\{4^{p-1}\gamma_1^p|Y_\tau-\bar{Y}_\tau|^p+4^{p-1}\gamma_2^p\E[|Y_s-\bar{Y}_s|^p]_{|s=\tau} \right. \right. \\ \left. \left. \qquad\qquad\qquad\qquad\qquad +4^{p-1}\kappa_1^p|Y_\sigma-\bar{Y}_\sigma|^p+4^{p-1}\kappa_2^p\E[|Y_s-\bar{Y}_s|^p]_{|s=\sigma}\right\} |\mathcal{F}_t\right].
\end{array}
\end{equation*}
Therefore,
$$
e^{p \beta t}|\overline{\Psi}(Y)_t-\overline{\Psi}(\bar{Y})_t|^p \le \underset{\tau\in\mathcal{T}_t}{\esssup\,} \E[G^1(\tau)|\mathcal{F}_t]+\underset{\sigma\in\mathcal{T}_t}{\esssup\,} \E[G^2(\sigma)|\mathcal{F}_t]:=V^1_t+V^2_t,
$$
where
\begin{equation}\begin{array}{lll}
G^1(\tau):=\int_{T-\delta}^{\tau}e^{p \beta s}2^{\frac{p}{2}-1} \eta^p \delta^{\frac{p-2}{2}} C_f^p\E[|Y_s-\bar{Y}_s|^p]ds  \\ \qquad\qquad\qquad  +2^{\frac{p}{2}-1}e^{p \beta \tau}(4^{p-1}\gamma_1^p|Y_{\tau}-\bar{Y}_{\tau}|^p+4^{p-1}\gamma_2^p\E[|Y_{s}-\bar{Y}_{s}|^p]_{|s=\tau}), \\
G^2(\sigma) :=2^{\frac{p}{2}-1} e^{p \beta \sigma}(4^{p-1}\kappa_1^p|Y_{\sigma}-\bar{Y}_{\sigma}|^p+4^{p-1}\kappa_2^p\E[|Y_{s}-\bar{Y}_{s}|^p]_{|s=\sigma}).
\end{array}
\end{equation}
which yields
\begin{equation}\label{est-1}
\underset{\tau\in\mathcal{T}_{T-\delta}}{\sup}\E[e^{p \beta \tau}|\overline{\Psi}(Y)_{\tau}-\overline{\Psi}(\bar{Y})_{\tau}|^p] \le \underset{\tau\in\mathcal{T}_{T-\delta}}{\sup}\E[V^1_{\tau}+V^2_{\tau}].
\end{equation}
We have
\begin{equation}\label{est-2}
\underset{\tau\in\mathcal{T}_{T-\delta}}{\sup}\E[V_{\tau}]\le \alpha \underset{\tau\in\mathcal{T}_{T-\delta}}{\sup}\E[e^{p \beta \tau}|Y_{\tau}-\bar{Y}_{\tau}|^p], 
\end{equation}
where $\alpha:=2^{\frac{p}{2}-1}\delta^{1+\frac{p-2}{2}} \eta^p C_f^p+2^{\frac{p}{2}-1}4^{p-1}(\gamma_1^p+\gamma_2^p+\kappa_1^p+\kappa_2^p)$ and $\textcolor{black}{V_t:=V_t^{1}+V_t^2}$.

\noindent \textcolor{black}{Let $\sigma \in \mathcal{T}_0$}. Indeed, by Lemma D.1 in \cite{KS98},  there exist sequences $(\tau^1_n)_n$ and $(\tau^2_n)_n$ of stopping times in $\mathcal{T}_{\sigma}$ such that
$$
V^1_{\sigma}=\underset{n\to\infty}{\lim}\E[G^1(\tau^1_n)|\mathcal{F}_{\sigma}]
$$
and 
$$
V^2_{\sigma}=\underset{n\to\infty}{\lim}\E[G^2(\tau^2_n)|\mathcal{F}_{\sigma}].
$$
Therefore, by Fatou's Lemma, we have
$$
\E[V^1_{\sigma}]+\E[V^2_{\sigma}]\le \underset{n\to\infty}{\underline{\lim}}\E[G^1(\tau^1_n)]+\underset{n\to\infty}{\underline{\lim}}\E[G^2(\tau^2_n)]\le \underset{\tau\in\mathcal{T}_{T-\delta}}{\sup}\E[G^1(\tau)]+\underset{\tau\in\mathcal{T}_{T-\delta}}{\sup}\E[G^2(\tau)].
$$
Using \eqref{Wass-property-1} and noting that $e^{p\beta\tau}\E[|Y_s-\bar{Y}_s|^p]_{|s=\tau}=\E[e^{p\beta s}|Y_s-\bar{Y}_s|^p]_{|s=\tau}$, we obtain
$$
\underset{\tau\in\mathcal{T}_{T-\delta}}{\sup}\E[G^1(\tau)]+\underset{\tau\in\mathcal{T}_{T-\delta}}{\sup}\E[G^2(\tau)] \le \alpha \underset{\tau\in\mathcal{T}_{T-\delta}}{\sup}\E[e^{p \beta \tau}|Y_{\tau}-\bar{Y}_{\tau}|^p]
$$
which in turn yields \eqref{est-2}.

Assuming $(\gamma_1,\gamma_2, \kappa_1, \kappa_2)$ satisfies
\begin{equation*}
\gamma_1^p+\gamma_2^p+\kappa_1^p+\kappa_2^p<4^{1-p}2^{1-\frac{p}{2}}
\end{equation*}
we can choose  
$$
0<\delta<\left(\frac{1}{2^{\frac{p}{2}-1} \eta^p C_f^p}
\left(1-4^{p-1}2^{\frac{p}{2}-1}(\gamma^p_1+\gamma^p_2+\kappa^p_1+\kappa^p_2)\right)\right)^{\frac{p}{2p-2}}
$$ 
 to make $\overline{\Psi}$ a contraction on $\mathbb{L}^p_\beta$ over the time interval $[T-\delta,T]$, i.e. $\overline{\Psi}$ admits a unique fixed point over $[T-\delta,T]$.\\

\medskip
\textit{\underline{Step 2} (Existence of the value of the game and link with the mean-field doubly reflected BSDE \eqref{BSDE1}).} Let $\overline{V} \in \mathbb{L}_\beta^p$ be the fixed point for $\overline{\Psi}$ obtained in \textit{Step 1} and $(\hat{Y},\hat{Z},\hat{U},\hat{K}^1, \hat{K}^2) \in \mathcal{S}^p \times \mathcal{H}^{p,1} \times \mathcal{H}^p_\nu \times (\mathcal{S}_{i}^p)^2$ be the unique solution of the standard doubly reflected BSDE, with barriers $h_1(s,\overline{V}_s,\P_{V_s})$ and $h_2(s,\overline{V}_s,\P_{\overline{V}_s})$ and driver $g(s,y,z,u):=f(s,y,z,u,\P_{\overline{V}_s})$, i.e.
$$ \begin{array}{lll}
\hat{Y}_t = \xi +\int_t^T f(s,\hat{Y}_{s},  \hat{Z}_s,\hat{U}_s, \P_{V_s})ds + (\hat{K}^1_T - \hat{K}^1_t)-(\hat{K}^2_T -\hat{K}^2_t) \\ \qquad\qquad\qquad\qquad -\int_t^T \hat{Z}_s dB_s  - \int_t^T \int_{\R^\star} \hat{U}_s(e) \tilde N(ds,de), \qquad\quad T-\delta \leq t \leq T.
\end{array}
$$
Then, by Theorem 4.1. in \cite{dqs16}, we have
$$
\hat{Y}_t= \overline{\Psi}(\overline{V})_t=\underline{\Psi}(\overline{V})_t,
$$
\textcolor{black}{which combined with \textit{Step 1}, gives $\hat{Y}_\cdot=\overline{V}_{\cdot}$ and $\overline{V}_{\cdot}=\underline{\Psi}(\overline{V})_{\cdot}$} Thus, $\overline{V}_{\cdot}=\underline{V}_{\cdot}=\hat{Y}_{\cdot}$. This relation also yields existence of a solution for \eqref{BSDE1} on $[T-\delta, T]$. Hence, by the uniqueness of the solution of a doubly reflected BSDE, we obtain uniqueness of the associated processes $(\hat{Z},\hat{U},\hat{K}^1, \hat{K}^2)$ and combined with the fixed point property of $V$, we derive existence and uniqueness of the solution of \eqref{BSDE1} on $[T-\delta,T]$.

\noindent Applying the same method as in \textit{Step 1} on each time interval $[T-(j+1)\delta,T-j\delta]$, $1 \leq j \leq m$, with the same operator $\overline{\Psi}$, but with terminal condition $Y_{T-j\delta}$ at time $T-j\delta$, we build recursively, for $j=1$ to $m$, a solution $(Y^j,Z^j,U^j,K^{1,j},K^{2,j})$. Pasting properly these processes, we obtain an unique solution $(Y,Z,U,K^{1},K^{2})$ of \eqref{BSDE1} on $[0,T]$.

\noindent By using again the relation between classical Dynkin games and doubly reflected BSDEs, we also get the existence of a \textit{value of the game} $V$, which satisfies $V_\cdot=Y_\cdot$, and therefore also belongs to  $\mathcal{S}^p$.
\qed
\end{proof}

We now introduce the definition of $S$-saddle points in our setting and provide sufficient conditions on the barriers which ensure the existence of saddle points.

\paragraph{Existence of a $S$-saddle point.} Assume that the mean-field Dynkin game admits a \textit{common value} $(V_t)$. The associated payoff is denoted by
$$P(\tau,\sigma):= h_1(\tau,V_\tau, \mathbb{P}_{V_\tau})\ind_{\{\tau \leq \sigma<T\}}+h_2(\tau,V_\tau, \mathbb{P}_{V_\tau})\ind_{\{\sigma < \tau\}}+\xi \ind_{\{\tau \wedge \sigma=T\}}.$$

We now give the definition of a $S$-saddle point for this game problem.

\begin{Definition}
Let $S \in \mathcal{T}_0$. A pair $(\tau^\star, \sigma^\star) \in (\mathcal{T}_S)^2$ is called an $S$-saddle point for the \textit{mean-field Dynkin game problem} if for each $(\tau, \sigma) \in (\mathcal{T}_S)^2$ we have
\begin{align}
    \mathcal{E}^{f \circ V}_{S, \tau \wedge \sigma^{\star}}(P(\tau,\sigma^\star)) \leq \mathcal{E}^{f \circ V}_{S, \tau^\star \wedge \sigma^{\star}}(P(\tau^\star,\sigma^\star)) \leq \mathcal{E}^{f \circ V}_{S, \tau^\star \wedge \sigma}(P(\tau^\star,\sigma)).
\end{align}
\end{Definition}

We now provide sufficient conditions which ensure the existence of an $S$-saddle point.

\begin{Theorem}[Existence of $S$-saddle points]\label{saddle}
Suppose that $\gamma_1$, $\gamma_2$, $\kappa_1$ and $\kappa_2$ satisfy the condition \eqref{ExistenceCond}. Assume that $h_1$ (resp. $h_2$) take the form $h_1(t,\omega,y,\mu):=\xi_t(\omega)+\kappa^1(y,\mu)$ (resp. $h_2(t,\omega,y,\mu):=\zeta_t(\omega)+\kappa^2(y,\mu)$), where $\xi$ and $-\zeta$  belong to $\mathcal{S}^p$ and are  left upper semicontinuous process along stopping times, and $\kappa^1$ (resp $\kappa^2$) are bounded and Lipschitz functions with respect to $(y,\mu)$. 
\noindent For each $S \in \mathcal{T}_0$, consider the pair of stopping times $(\tau_{S}^{\star}, \sigma_{S}^\star)$ defined by
\begin{align}\label{S}
    \tau_{S}^\star:=\inf \{t \geq S:\,\, V_t=h_1(t,V_t,\mathbb{P}_{V_t})\}\,\,\, \text{ and } \,\,\,\sigma_{S}^\star:=\inf \{t \geq S:\,\, V_t=h_2(t,V_t,\mathbb{P}_{V_t})\}.
\end{align}
Then the pair of stopping time $(\tau_{S}^\star, \sigma_{S}^\star)$ given by $\eqref{S}$ is an $S$-saddle point.
\end{Theorem}
\begin{proof}

Consider the following iterative scheme. Let $V^{(0)} \equiv 0$ (with $\P_{V^{(0)}}=\delta_0$) be the starting point and define
$$V^{(m)}:=\overline{\Psi}(V^{(m-1)}),\quad m\ge 1.$$
By applying the results on standard doubly reflected BSDEs and their relation with classical Dynkin games, we obtain that, for $1 \le i \le m$, $V^{(i)}$ coincides with the component $Y$ of the solution of the doubly reflected BSDE associated with \textcolor{black}{$f \circ V^{(i-1)}$} and obstacles $h_1(t,V^{(i-1)}_t,\mathbb{P}_{V^{(i-1)}_t})$ and $h_2(t,V^{(i-1)}_t,\mathbb{P}_{V^{(i-1)}_t})$. Due to the assumptions on $h_1$ and $h_2$, \textcolor{black}{we get that $V^{(1)}$ admits only jumps at totally inaccessible stopping times, and by induction, the same holds for $V^{(i)}$, for all $i$}.

Since the condition $\eqref{ExistenceCond}$ is satisfied, by Theorem \ref{Dynkin}, the sequence $V_t^{(m)}$ is Cauchy for the norm $\mathbb{L}_\beta^{p}$ and therefore converges in $\mathbb{L}_\beta^{p}$ to the fixed point of the map $\overline{\Psi}$.\\
Let $\tau \in \mathcal{T}_0$ be a predictable stopping time. Since $\Delta V^{(m)}_\tau=0$ a.s. for all $m$, we obtain
\begin{align}
    \mathbb{E}\left[|\Delta V_\tau|^p\right]=\mathbb{E}\left[|\Delta V_\tau-\Delta V^{(m)}_\tau|^p\right] \leq 2^p \sup_{\tau \in \mathcal{T}_0}\mathbb{E}\left[|V_\tau-V^{(m)}_\tau|^p\right],
\end{align}
which implies that $\Delta V_\tau=0$ a.s. Therefore, $h_1(t,V_t,\mathbb{P}_{V_t})$ and $-h_2(t,V_t,\mathbb{P}_{V_t})$ are left upper semicontinuous along stopping times. By Theorem 3.7 (ii) in \cite{dqs16}, $(\tau_{S}^\star, \sigma_{S}^\star)$ given by $\eqref{S}$ is a $S$-saddle point.
\end{proof}

\medskip

\medskip

\bigskip

\section{Weakly interacting zero-sum Dynkin games}\label{WIDG}
\noindent We now interpret the mean-field Dynkin game studied in the previous section at the particle level and introduce the \textit{weakly interacting zero-sum Dynkin games}.\\

\noindent Given a vector $\textbf{x} := (x^1, \ldots, x^n) \in \R^n$, denote by
$$
L_n [\textbf{x}] := \frac{1}{n} \sum_{k=1}^n \delta_{x^k},
$$
the empirical measures associated to $\bf x$. 

\medskip
\noindent Let $\{B^i\}_{1\leq i \leq n}$, $\{\tilde N^i\}_{1\leq i \leq n}$ be independent copies of $B$ and $\tilde N$ and denote by $\mathbb{F}^n:=\{\mathcal{F}_t^n\}_{t \in [0,T]}$ the completion of the filtration generated by $\{B^i\}_{1\leq i \leq n}$ and $\{\tilde N^i\}_{1\leq i \leq n}$. Thus, for each $1\le i\le n$, the filtration generated by $(B^i,\tilde{N}^i)$ is a sub-filtration of $\mathbb{F}^n$, and we denote by $\widetilde{\mathbb{F}}^i:=\{\widetilde{\mathcal{F}}^i_t\}$ its completion. Let $\mathcal{T}^n_t$ be the set of $\mathbb{F}^n$ stopping times with values in $[0,T]$.\\

\noindent In this section, we make the following additional assumption.
\begin{Assumption}\label{Assump-PS} We assume that
\begin{itemize} 
\item $f$, $h_1$ and $h_2$ are deterministic functions of $(t,y,z,u,\mu)$ and $(t,y,\mu)$, respectively. 
\item  $\xi^{i,n}\in L^p(\mathcal{F}_T^n), \quad i=1,\ldots,n$.
\item  $h_1(T,\xi^{i,n},L_n[{\bf \xi}^n]) \le \xi^{i,n}\le h_2(T,\xi^{i,n},L_n[{\bf \xi}^n])\quad \P\text{-a.s.},\quad i=1,\ldots,n.$
\end{itemize}
\end{Assumption}

Since we will use the whole filtration $\mathbb{F}^n$, the first assumption is imposed to avoid unnecessarily adaptedness problems which might make parts of the proofs heavier. 

\medskip

\noindent Endow the product spaces $\spn_\beta:=\mathcal{S}_\beta^p\times \mathcal{S}_\beta^p \times \cdots \times \mathcal{S}_\beta^p$  and 
$\lpn_{\beta}:=\mathbb{L}_{\beta}^p\times \mathbb{L}_{\beta}^p \times \cdots \times \mathbb{L}_{\beta}^p$
with the respective norms
\begin{equation}\label{n-norm-2}
\|h\|^p_{\spn_\beta}:=\sum_{1 \le i\le n}\|h^i\|^p_{\mathcal{S}_\beta^p},\qquad \|h\|^p_{\lpn_\beta}:=\sum_{1 \le i\le n}\|h^i\|^p_{\mathbb{L}_\beta^p}.
\end{equation}
\noindent Note that $\spn_\beta$ and $\lpn_\beta$ are complete metric spaces. We denote by $\spn:=\spn_0,\,\, \lpn:=\lpn_0$.\\
Consider the applications
$$
\begin{array}{lll}
\textbf{f}^{i} \circ \textbf{Y}^n:[0,T] \times \Omega \times \R \times \R^n \times (\R^*)^n \mapsto \R \\  \qquad\qquad (\textbf{f}^{i} \circ \textbf{Y}^n)(t,\omega, y,z,u)=f(t, y,z^i,u^i,L_{n}[\textbf{Y}^n_{t}](\omega)), \quad i=1,\ldots,n.
\end{array}
$$

\noindent We now introduce the notions of \textit{lower value} and \textit{upper value} of \textit{interacting Dynkin games}.

\begin{Definition} Let 
$\underline{\Psi}: \,\lpn_\beta\longrightarrow \lpn_\beta$ to be the mapping that associates to a process $\textbf{Y}^n:=(Y^{1,n},Y^{2,n},\ldots,Y^{n,n})$ the process $\underline{\Psi}(\textbf{Y}^n)=(\underline{\Psi}^1(\textbf{Y}^n),\underline{\Psi}^2(\textbf{Y}^n),\ldots, \underline{\Psi}^n(\textbf{Y}^n))$ defined by the following system: for every $i=1,\ldots n$ and $t\le T$,
\begin{equation}\label{snell-i-2}
\begin{array}{lll} 
\underline{\Psi}^i(\Ybf^n)_t := \underset{\tau\in\mathcal{T}^n_t}{\esssup\,} \underset{\sigma \in \mathcal{T}^n_t}{\essinf\,} \mathcal{E}^{\textbf{f}^{i} \circ \textbf{Y}^n}_{t,\tau \wedge \sigma}\left[ h_1(\tau,Y^{i,n}_{\tau},L_{n}[\textbf{Y}^n_{s}]|_{s=\tau})1_{\{\tau \leq \sigma<T\}} \right. \\ \left. \qquad\qquad\qquad\qquad\qquad\qquad\qquad\qquad +h_2(\sigma,Y^{i,n}_{\sigma},L_{n}[\textbf{Y}^n_{s}]|_{s=\sigma})1_{\{\sigma < \tau\}}  +\xi^{i,n}1_{\{\tau \wedge \sigma=T\}}\right].
\end{array}
\end{equation}
We define the \textit{lower value function} of the \textit{interacting Dynkin game}, denoted by $\underline{\textbf{V}}^n$, as the \textit{fixed point} of the application $\underline{\Psi}$, that is $\underline{\Psi}^i(\underline{\textbf{V}}^n)=\underline{V}^{i,n}$, for all $1 \leq i \leq n$.
\end{Definition}

\begin{Definition} Let 
$\overline{\Psi}: \,\lpn_\beta\longrightarrow \lpn_\beta$ to be the mapping that associates to a process $\textbf{Y}^n:=(Y^{1,n},Y^{2,n},\ldots,Y^{n,n})$ the process $\overline{\Psi}(\textbf{Y}^n)=(\overline{\Psi}^1(\textbf{Y}^n),\overline{\Psi}^2(\textbf{Y}^n),\ldots, \overline{\Psi}^n(\textbf{Y}^n))$ defined by the following system: for every $i=1,\ldots n$ and $t\le T$,
\begin{equation}\label{snell-i-2}
\begin{array}{lll} 
\overline{\Psi}^i(\Ybf^n)_t := \underset{\sigma \in \mathcal{T}^n_t}{\essinf\,}\underset{\tau\in\mathcal{T}^n_t}{\esssup\,} \mathcal{E}^{\textbf{f}^{i} \circ \textbf{Y}^n}_{t,\tau \wedge \sigma}\left[ h_1(\tau,Y^{i,n}_{\tau},L_{n}[\textbf{Y}^n_{s}]|_{s=\tau})1_{\{\tau \leq \sigma<T\}} \right. \\ \left. \qquad\qquad\qquad\qquad\qquad\qquad\qquad\qquad +h_2(\sigma,Y^{i,n}_{\sigma},L_{n}[\textbf{Y}^n_{s}]|_{s=\sigma})1_{\{\sigma < \tau\}}  +\xi^{i,n}1_{\{\tau \wedge \sigma=T\}}\right].
\end{array}
\end{equation}
We define the \textit{upper value function} of the \textit{interacting Dynkin game}, denoted by $\overline{\textbf{V}}^n$, as the \textit{fixed point} of the application $\overline{\Psi}$, that is $\overline{\Psi}^i(\overline{\textbf{V}}^n)=\overline{V}^{i,n}$, for all $1 \leq i \leq n$.
\end{Definition}

We are now in position to provide the definition of a \textit{common value} for this type of game.
\begin{Definition}
The interacting Dynkin game is said to admit a \textit{common value}, called the \textit{value of the game} and denoted by $\textbf{V}^n$,
if $\underline{\textbf{V}}^n$ and $\overline{\textbf{V}}^n$ exist and $V^{i,n}_t=\underline{V}^{i,n}_t=\overline{V}^{i,n}_t$ for all $t \in [0,T]$ and for all $1 \leq i \leq n$. 
\end{Definition}
More precisely, the value of the game, denoted by $\textbf{V}^n$, corresponds to the common fixed point of the applications $\overline{\Psi}$ and $\underline{\Psi}$, i.e. it satisfies $V^{i,n}_t=\overline{\Psi}^{i}(\textbf{V}^n)_t=\underline{\Psi}^{i}(\textbf{V}^n)_t$.
The main result of this section is to give conditions under which the game admits a value and to obtain the characterization of the common value through the following system of interacting doubly reflected BSDEs, which will be shown to have a unique solution.
\begin{align}\label{BSDEParticle}
    \begin{cases}
    Y^{i,n}_t = \xi^{i,n} +\int_t^T f(s,Y^{i,n}_{s},Z^{i,n}_{s},U^{i,n}_{s}, L_n[\textbf{Y}^n_{s}])ds + K^{1,i,n}_T - K^{1,i,n}_t+ K^{2,i,n}_t-K^{2,i,n}_T  \\
    	\qquad\qquad \qquad  - \int_t^T Z^{i,n}_s dB_s -  \int_t^T \int_{\R^\star} U^{i,n}_s(e) \tilde N(ds,de), \quad t \in [0,T],\\
    h_2(t,Y^{i,n}_{t},L_n[\textbf{Y}^n_{t}]) \geq Y^{i,n}_{t} \geq h_1(t,Y^{i,n}_{t},L_n[\textbf{Y}^n_{t}]),  \quad  t \in [0,T], \\
    \int_0^T (Y^{i,n}_{t-}-h_1(t,Y^{i,n}_{t-},L_n[\textbf{Y}^n_{t^-}]))dK^{1,i,n}_t = 0,\\ \int_0^T (Y^{i,n}_{t-}-h_2(t,Y^{i,n}_{t-},L_n[\textbf{Y}^n_{t^-}]))dK^{2,i,n}_t = 0,\\ dK_t^{1,i,n} \perp dK_t^{2,i,n}.
    \end{cases}
\end{align}



Note that we have the inequality
\begin{equation}\label{ineq-wass-1}
\mathcal{W}_p^p(L_n[\textbf{x}], L_n[\textbf{y}]) \le \frac{1}{n}\sum_{j=1}^{n}|x_j-y_j|^p,
\end{equation}
in particular,
\begin{equation}\label{ineq-wass-2}
\mathcal{W}_p^p(L_n[\textbf{x}], L_n[\textbf{0}]) \le \frac{1}{n}\sum_{j=1}^{n}|x_j|^p.
\end{equation}

\begin{Theorem}[\textit{Existence of the value and link with interacting system of doubly reflected BSDEs}]\label{Dynkin2}
Suppose that Assumptions \ref{generalAssump} and \ref{Assump-PS} are in force for some $p\ge 2$. Assume that $\gamma_1$, $\gamma_2$, $\kappa_1$ and $\kappa_2$ satisfy
\begin{align*}\label{smallnessCond-new}
\gamma_1^p+\gamma_2^p+\kappa_1^p+\kappa_2^p<2^{3-\frac{5p}{2}}.
\end{align*}
Then,
\begin{itemize}
\item[(i)]the \textit{interacting Dynkin game} admits a value $\textbf{V}^n \in \mathcal{S}^{p, \otimes n}$,
\item[(ii)] the interacting system of doubly reflected BSDEs \eqref{BSDEParticle} has an unique solution $(\textbf{Y}^n,\textbf{Z}^n,\textbf{U}^n,\textbf{K}^{1,n},\textbf{K}^{2,n})$ in $\mathcal{S}^{p,\otimes n} \otimes \mathcal{H}^{p,n\otimes n} \otimes \mathcal{H}_{\nu}^{p,n\otimes n} \otimes \mathcal{S}_{i}^{p,\otimes n}$,
\item[(iii)] We have $V_\cdot^{i,n}=Y_\cdot^{i,n}$, for all $1 \leq i \leq n$.
\end{itemize}
\end{Theorem}

\begin{proof}

\uline{\textit{Step 1.}} $($\textit{Well-posedness and contraction property of the operators $\overline{\Psi}$ and $\underline{\Psi})$. The well-posedness of the operators $\overline{\Psi}$ and $\underline{\Psi}$ follows by similar arguments to those used in {\it Step 1} of the proof of Theorem \eqref{Dynkin}. We now only show that $\overline{\Psi}$ is a contraction on the time interval $[T-\delta,T]$, for some well chosen $\delta$. The same proof holds for the operator $\underline{\Psi}$}.

\medskip
Fix $\mathbf{Y}^n=(Y^{1,n},\dots,Y^{n,n}),\bar{\mathbf{Y}}^n=(\bar Y^{1,n},\dots,\bar Y^{n,n}) \in \lpn_\beta$, $(\hat{Y},\tilde{Y}) \in (\mathcal{S}_\beta^{p})^2$, $(\hat{Z},\tilde{Z}) \in (\mathcal{H}^{p,1})^2$, $(\hat{U},\tilde{U}) \in (\mathcal{H}_\nu^{p})^2$. By the Lipschitz continuity of $f$, $h_1$ and $h_2$, we get
  \begin{eqnarray}\label{f-h-lip-4}
&&|f(s,\hat{Y}_s,\hat{Z}_s,\hat{U}_s, L_{n}[\textbf{Y}^n_s])-f(s,\tilde{Y}_s,\tilde{Z}_s,\tilde {U}_s,L_{n}[\bar{\textbf{Y}}^n_s])| \le  C_f(|\hat{Y}_s-\tilde{Y}_s|+|\hat{Z}_s-\tilde{Z}_s|\nonumber \\
&&\qquad \qquad\qquad \qquad \qquad\qquad +|\hat{U}_s-\tilde{U}_s|_\nu +\mathcal{W}_p(L_{n}[\textbf{Y}^n_s], L_{n}[\bar{\textbf{Y}}^n_s])), \nonumber\\ 
&&|h_1(s,Y^{i,n}_s,L_{n}[\textbf{Y}^n_s])-h_1(s,\bar{Y}^{i,n}_s,L_{n}[\bar{\textbf{Y}}^n_s])| \le  \gamma_1|Y^{i,n}_s-\bar{Y}^{i,n}_s|+\gamma_2\mathcal{W}_p(L_{n}[\textbf{Y}^n_s], L_{n}[\bar{\textbf{Y}}^n_s]), \nonumber\\
&& |h_2(s,Y^{i,n}_s,L_{n}[\textbf{Y}^n_s])-h_2(s,\bar{Y}^{i,n}_s,L_{n}[\bar{\textbf{Y}}^n_s])| \le  \kappa_1|Y^{i,n}_s-\bar{Y}^{i,n}_s|+\kappa_2\mathcal{W}_p(L_{n}[\textbf{Y}^n_s], L_{n}[\bar{\textbf{Y}}^n_s]).
\end{eqnarray}
By \eqref{ineq-wass-1}, we have
\begin{equation}\label{ineq-wass-1-1-3}
\mathcal{W}_p^p(L_{n}[\textbf{Y}^n_s], L_{n}[\bar{\textbf{Y}}^n_s])) \le \frac{1}{n}\sum_{j=1}^n |Y_s^{j,n}-\bar{Y}_s^{j,n}|^p.
\end{equation}
Then, using the estimates from Proposition A.1 (see the appendix), for any $t\leq T$ and $i=1,\ldots, n$, we have  
\begin{equation*}\begin{array} {lll}
|\overline{\Psi}^i(\Ybf^n)_t-\overline{\Psi}^i(\bar{\Ybf}^n)_t|^{p} \\
\le \underset{\tau\in\mathcal{T}^n_t}{\esssup\,} \underset{\sigma \in\mathcal{T}^n_t}{\esssup\,}\left|\mathcal{E}_{t,\tau \wedge \sigma}^{\textbf{f}^{i} \circ \Ybf}[h_1(\tau,Y^{i,n}_\tau,L_{n}[\Ybf^n_s]_{s=\tau})\ind_{\{\tau \leq \sigma \}} \right. \\ \left. \qquad\qquad \qquad\qquad \qquad\qquad +h_2(\sigma,Y^{i,n}_\sigma,L_{n}[\Ybf^n_s]_{s=\sigma})\ind_{\{\sigma<\tau\}}+\xi^{i,n}\ind_{\{\sigma \wedge \tau=T\}}] \right. \\ \left. \qquad\qquad - \mathcal{E}_{t,\tau \wedge \sigma}^{\textbf{f}^{i} \circ \bar{\Ybf}^n }[h_1(\tau,\bar{Y}^{i,n}_\tau,L_{n}[\bar{\Ybf}^n_s]_{s=\tau})1_{\{\tau \leq \sigma\}}+h_2(\sigma,\bar{Y}^{i,n}_\sigma,L_{n}[\bar{\Ybf}^n_s]_{s=\sigma})1_{\{\sigma< \tau\}}+\xi^{i,n}\ind_{\{\tau \wedge \sigma=T\}}] \right|^{p}\\
\le \underset{\tau\in\mathcal{T}^n_t}{\esssup\,}\underset{\sigma\in\mathcal{T}^n_t}{\esssup\,}\eta^p 2^{\frac{p}{2}-1}\E\left[\left(\int_t^{\tau \wedge \sigma} e^{2 \beta (s-t)}  |(\textbf{f}^{i} \circ \Ybf^n) (s,\widehat{Y}^{i,\tau, \sigma}_s,\widehat{Z}^{i,\tau, \sigma}_s,\widehat{U}^{i,\tau, \sigma}_s) \right.\right. \\ \left. \left. \qquad\qquad\qquad
-(\textbf{f}^{i} \circ \bar{\Ybf}^n) (s,\widehat{Y}^{i,\tau, \sigma}_s, \widehat{Z}^{i,\tau, \sigma}_s,\widehat{U}^{i,\tau, \sigma}_s)|^2 ds \right)^{\frac{p}{2}}  \right. \\ \left. \qquad\qquad \qquad\qquad +2^{\frac{p}{2}-1} e^{p \beta(\tau \wedge \sigma-t)}|\left(h_1(\tau,Y^{i,n}_\tau,L_{n}[\Ybf^n_s]_{s=\tau})-h_1(\tau,\bar{Y}^{i,n}_\tau, L_{n}[\bar{\Ybf}^n_s]_{s=\tau})\right)\ind_{\{\tau \leq \sigma<T\}} \right. \\ \left. \qquad\qquad\qquad\qquad\qquad +\left(h_2(\sigma,Y^{i,n}_\sigma,L_{n}[\Ybf^n_s]_{s=\sigma})-h_2(\sigma,\bar{Y}^{i,n}_\sigma, L_{n}[\bar{\Ybf}^n_s]_{s=\sigma})\right)\ind_{\{\sigma < \tau\}}|^p
|\mathcal{F}_t\right] \\
= \underset{\tau\in\mathcal{T}^n_t}{\esssup\,} \underset{\sigma \in\mathcal{T}^n_t}{\esssup\,} \eta^p 2^{\frac{p}{2}-1}\E\left[\left(\int_t^{\tau \wedge \sigma} e^{2 \beta (s-t)} |f(s,\widehat{Y}^{i,\tau, \sigma}_s,\widehat{Z}^{i,i,\tau, \sigma}_s,\widehat{U}^{i,i,\tau, \sigma}_s,
L_{n}[\Ybf^n_s])  \right.\right. \\ \left. \left. \qquad\qquad\qquad -f(s,\widehat{Y}^{i,\tau,\sigma}_s,\widehat{Z}^{i,i,\tau, \sigma}_s,\widehat{U}^{i,i,\tau, \sigma}_s, L_{n}[\bar{\Ybf}^n_s])|^2 ds \right)^{p/2}\right. \\ \left. \qquad\qquad \qquad\qquad +2^{\frac{p}{2}-1} e^{p \beta(\tau \wedge \sigma-t)}|\left(h_1(\tau,Y^{i,n}_\tau,L_{n}[\Ybf^n_s]_{s=\tau})-h_1(\tau,\bar{Y}^{i,n}_\tau, L_{n}[\bar{\Ybf}^n_s]_{s=\tau})\right)\ind_{\{\tau \leq \sigma<T\}} \right. \\ \left. \qquad\qquad\qquad \qquad\qquad +\left(h_2(\sigma,Y^{i,n}_\sigma,L_{n}[\Ybf^n_s]_{s=\sigma})-h_2(\sigma,\bar{Y}^{i,n}_\sigma, L_{n}[\bar{\Ybf}^n_s]_{s=\sigma})\right)\ind_{\{\sigma < \tau\}}|^p
|\mathcal{F}_t\right],
 \end{array}
\end{equation*}
where $(\widehat{Y}^{i,\tau, \sigma},\widehat{Z}^{i,\tau, \sigma},\widehat{U}^{i,\tau, \sigma})$ is the solution of the BSDE associated with driver $\textbf{f}^{i}\circ \bar{\Ybf}^n$, terminal time $\tau \wedge \sigma$ and terminal condition $h_1(\tau,\bar{Y}^{i,n}_\tau,L_{n}[\bar{\Ybf}^n_s]_{s=\tau})\ind_{\{\tau \leq \sigma<T\}}+h_2(\sigma,\bar{Y}^{i,n}_\sigma,L_{n}[\bar{\Ybf}^n_s]_{s=\sigma})\ind_{\{\sigma < \tau\}}+\xi^{i,n}\ind_{\{\tau \wedge \sigma=T\}}$.

\noindent Therefore, using \eqref{f-h-lip-4} and \eqref{ineq-wass-1-1-3}, we have, for any $t\le T$ and $i=1,\ldots, n$,

\begin{equation}\label{ineq-evsn-2}\begin{array} {lll}
e^{p \beta t}|\overline{\Psi}^i(\Ybf^n)_t-\overline{\Psi}^i(\bar{\Ybf}^n)_t|^p \\ \qquad\qquad \le \underset{\tau  \in\mathcal{T}^n_t}{\esssup\,} \underset{\sigma \in\mathcal{T}^n_t}{\esssup\,}\E\left[ \int_t^{\tau \wedge \sigma} 2^{\frac{p}{2}-1} \delta^{\frac{p-2}{2}} \eta^{p}C_f^{p}\left( \frac{1}{n}\sum_{j=1}^n e^{p\beta s}|Y_s^{j,n}-\bar{Y}_s^{j,n}|^p \right) ds  \right.  \\ \left.   \qquad \qquad\qquad\qquad +2^{\frac{p}{2}-1} \left(\gamma_1 e^{\beta \tau}|Y^{i,n}_{\tau}-\bar{Y}^{i,n}_{\tau}|+ \gamma_2 \left(\frac{1}{n}\sum_{j=1}^ne^{p\beta \tau}|Y_{\tau}^{j,n}-\bar{Y}_{\tau}^{j,n}|^p \right)^{\frac{1}{p}} \right.\right. \\ \left. \left.  \qquad\qquad\qquad \qquad\qquad +\kappa_1 e^{\beta \sigma}|Y^{i,n}_{\sigma}-\bar{Y}^{i,n}_{\sigma}|+ \kappa_2 \left(\frac{1}{n}\sum_{j=1}^ne^{p\beta \sigma}|Y_{\sigma}^{j,n}-\bar{Y}_{\sigma}^{j,n}|^p \right)^{\frac{1}{p}}\right)^{p}|\mathcal{F}_t\right].
\end{array}
\end{equation} 
Next, let $\delta >0$, $t\in [T-\delta,T]$ and 
\begin{align*}
\alpha:&=\max(\delta^{1+\frac{p-2}{2}} 2^{\frac{p}{2}-1} \eta^{p}C^{p}_f+2^{\frac{p}{2}-1} 4^{p-1}\gamma^p_1, \delta^{1+\frac{p-2}{2}} \eta^{p}C^{p}_f+2^{\frac{p}{2}-1} 4^{p-1}\gamma^p_2, \nonumber \\ & \qquad\qquad \delta^{1+\frac{p-2}{2}} \eta^{p}C^{p}_f+2^{\frac{p}{2}-1} 4^{p-1}\kappa^p_1, \delta^{1+\frac{p-2}{2}} \eta^{p}C^{p}_f+2^{\frac{p}{2}-1} 4^{p-1}\kappa^p_2).
\end{align*}
 By applying the same arguments as in the previous section, and using the definition \eqref{n-norm-2} we obtain

\begin{equation*}\begin{array} {lll}
\|\overline{\Psi}(\Ybf^n)-\overline{\Psi}(\bar{\Ybf}^n)\|^p_{\lpn_\beta[T-\delta,T]}\le \alpha \|\Ybf^n-\bar{\Ybf}^n\|^p_{\lpn_\beta[T-\delta,T]}.
\end{array}
\end{equation*}
Therefore, if $\gamma_1, \gamma_2, \kappa_1$ and $\kappa_2$ satisfy
$$
\gamma_1^p+\gamma_2^p+\kappa_1^p+\kappa_2^p<4^{1-p} 2^{1-\frac{p}{2}},
$$
we can choose 
$$
0<\delta<\left(\frac{1}{2^{\frac{p}{2}-1}\eta^p C_f^p}\left(1-4^{p-1}2^{\frac{p}{2}-1}(\gamma^p_1+\gamma^p_2+\kappa_1^p+\kappa_2^p)\right)\right)^{\frac{p}{2p-2}}
$$ 
to make $\overline{\Psi}$ a contraction on $\lpn_\beta([T-\delta,T])$, i.e. $\overline{\Psi}$ admits a unique fixed point over $[T-\delta,T]$. Moreover, in view of Assumption \ref{generalAssump} (ii)(a), it holds that $\Ybf^n\in \spn([T-\delta,T])$.\\

\medskip
\noindent \underline{\textit{Step 2.}}(\textit{Existence of a value of the interacting Dynkin game and link with interacting doubly reflected BSDEs}). Let us now show that the game admits a \textit{common value} on $[0,T]$ and that the interacting system of doubly reflected BSDE \eqref{BSDEParticle} has a unique solution, which is related to the value of the interacting Dynkin game.\\ 
\noindent Let $\textbf{V}^n$ be the fixed point associated with the map $\overline{\Psi}$ on $[T-\delta,T]$ obtained in \textit{Step 1}. By classical results on doubly reflected BSDEs (see e.g. \cite{dqs16}), for each $i$ between $1$ and $n$, there exists a unique solution $(Y^{i,n}, Z^{i,n}, U^{i,n}, K^{1,i,n}, K^{2,i,n})$ on $[T-\delta, T]$ of the doubly reflected BSDE associated with driver $\textbf{f}^{i} \circ \textbf{V}^n$ and obstacles $h_1(t,V^{i,n}_{t},L_n[\textbf{V}^n_{t}])$ and $h_2(t,V^{i,n}_{t},L_n[\textbf{V}^n_{t}])$. Furthermore, by using the relation between classical Dynkin games and doubly reflected BSDEs (see Theorem 4.10 in \cite{dqs16}), as well as the fixed point property of $\overline{\Psi}$ and $\underline{\Psi}$ showed above, we get that $Y^{i,n}_t=V_t^{i,n}$, for $t \in [T-\delta, T]$ and $1 \leq i \leq n$.

Applying the same method as in \textit{Step 1} on each time interval $[T-(j+1)\delta,T-j\delta]$, $1 \leq j \leq m$, with the same operator $\overline{\Psi}$, but terminal condition $\textbf{Y}^{i,n}_{T-j\delta}$ at time $T-j\delta$, we build recursively, for $j=1$ to $n$, a solution $(\textbf{Y}^{n,(j)},\textbf{Z}^{n,(j)},\textbf{U}^{n,(j)},\textbf{K}^{1,n,(j)}, \textbf{K}^{2,n,(j)})$, on each time interval $[T-(j+1)\delta,T-j\delta].$ Pasting properly these processes, we obtain a solution $(\textbf{Y}^{n},\textbf{Z}^{n},\textbf{U}^{n},\textbf{K}^{1,n}, \textbf{K}^{2,n})$ of \eqref{BSDEParticle} on $[0,T]$. The uniqueness of the solution of the system \eqref{BSDEParticle} is obtained by the fixed point property of $\textbf{Y}^n$ and the uniqueness of the associated processes  $(\textbf{Z}^{n},\textbf{U}^{n},\textbf{K}^{1,n}, \textbf{K}^{2,n})$, which follows by the uniqueness of the solution of standard doubly reflected BSDEs.  Moreover, using again the relation between standard Dynkin games with doubly reflected BSDEs, we obtain that \textit{the value of the game} $\textbf{V}^n$ exists on the full time interval $[0,T]$ and, furthermore, $V_\cdot^{i,n}={Y}_\cdot^{i,n}$. 
\end{proof}

\paragraph{Existence of an $S$-saddle point.} Assume that the \textit{interacting Dynkin game} admits a \textit{common value} which is denoted by $\textbf{V}^n_t$. We introduce the sequence of payoffs $\left\{P^{i,n}\right\}_{1 \leq i \leq n}$, which, for $1 \leq i \leq n$, is given by
$$P^{i,n}(\tau,\sigma):= h_1(\tau,V^{i,n}_\tau, L_n[\textbf{V}^n_\tau])\ind_{\{\tau \leq \sigma<T\}}+h_2(\sigma,V^{i,n}_\sigma, L_n[\textbf{V}^n_\sigma])\ind_{\{\sigma < \tau\}}+\xi^{i,n} \ind_{\{\tau \wedge \sigma=T\}}.$$

\medskip
We now give the definition of an $S$-saddle point for the \textit{interacting Dynkin game problem}.

\begin{Definition}
Let $S \in \mathcal{T}_0$. The sequence of pairs of stopping times $(\tau^{\star,i,n}, \sigma^{\star,i,n}) \in \mathcal{T}^n_S \times \mathcal{T}^n_S$ is called a \textit{system of $S$-saddle points} for the \textit{interacting Dynkin game problem} if for each $1\le i\le n$ and for each $(\tau, \sigma) \in (\mathcal{T}_S)^2$ we have $\P$-a.s.
\begin{align}
    \mathcal{E}^{\textbf{f}^{i} \circ \textbf{V}^n}_{S, \tau \wedge \sigma^{\star,i,n}}(P^{i,n}(\tau,\sigma^{\star,i,n})) \leq \mathcal{E}^{\textbf{f}^{i} \circ \textbf{V}^n}_{S, \tau^{\star,i,n} \wedge \sigma^{\star,i,n}}(P^{i,n}(\tau^{\star,i,n},\sigma^{\star,i,n})) \leq \mathcal{E}^{\textbf{f}^{i} \circ \textbf{V}^n}_{S, \tau^{\star,i,n} \wedge \sigma}(P^{i,n}(\tau^{\star,i,n},\sigma)).
\end{align}
\end{Definition}

\medskip
In the next theorem, we provide sufficient conditions which ensure the existence of \textit{a system of $S$-saddle points}.

\begin{Theorem}[Existence of a system of $S$-saddle points]\label{saddle1}
Suppose that $\gamma_1$, $\gamma_2$, $\kappa_1$ and $\kappa_2$ satisfy the condition \eqref{ExistenceCond}. Assume that $h_1$ (resp. $h_2$) take the form $h_1(t,\omega,y,\mu):=\xi_t(\omega)+\kappa^1(y,\mu)$ (resp. $h_2(t,\omega,y,\mu):=\zeta_t(\omega)+\kappa^2(y,\mu)$), where $\xi$ and $-\zeta$  belong to $\mathcal{S}^p$ and are  left upper semicontinuous process along stopping times, and $\kappa^1$ (resp $\kappa^2$) are bounded and Lipschitz functions with respect to $(y,\mu)$. 

\noindent For each $S \in \mathcal{T}_0$, consider the system of pairs of stopping times $(\tau_{S}^{\star,i,n}, \sigma_{S}^{\star,i,n})$ defined by
\begin{eqnarray}\label{S} \begin{array}{lll}
    \tau_{S}^{\star,i,n}:=\inf \{t \geq S:\, V^{i}_t=h_1(t,V^{i}_t,L_n[\textbf{V}^n_t])\},\\ \sigma_{S}^{\star,i,n}:=\inf \{t \geq S:\, V^{i}_t=h_2(t,V^{i}_t,L_n[\textbf{V}^n_t])\}.
    \end{array}
\end{eqnarray}
Then the sequence of pairs of stopping time $(\tau_{S}^{\star,i,n}, \sigma_{S}^{\star,i,n})$ given by $\eqref{S}$ is a system of  $S$-saddle points.
\end{Theorem}
\begin{proof}
Consider the following iterative algorithm. Let $\textbf{V}^{(0),n} \equiv 0$ be the starting point and define
$$\textbf{V}^{(m),n}:=\overline{\Psi}(\textbf{V}^{(m-1),n}),$$
where the inequality is understood component by component.
By using classical results on doubly reflected BSDEs and their relation with classical Dynkin games, we obtain that, for $1 \le i \le n$, $V^{(m),i,n}$ coincides with the component $Y^{i,n}$ of the solution of the doubly reflected BSDE associated with $\textbf{f}^{i} \circ \textbf{V}^{(m-1),n}$ and obstacles $h_1(t,V^{(m-1),i,n}_t,L_n[\textbf{V}^{(m-1),n}_t])$ and $h_2(t,V^{(m-1),i,n}_t,L_n[\textbf{V}^{(m-1),n}_t])$. Due to the assumptions on $h_1$ and $h_2$, we get that, for each $i$, $V^{(1),i,n}$ only admits jumps at totally inaccessible stopping times. By induction, the same holds for $V^{(m)}$, for all $m$.

Since the condition $\eqref{ExistenceCond}$ is satisfied, by Theorem \ref{Dynkin2}, the sequence $\textbf{V}_t^{(m),n}$ is a Cauchy sequence for the norm $\mathbb{L}_\beta^{p, \otimes n}$ and therefore converges in $\mathbb{L}_\beta^{p, \otimes n}$ to the fixed point of the map $\overline{\Psi}$.\\
Let $\tau \in \mathcal{T}_0$ be a predictable stopping time. Since $\Delta V^{(m),i,n}_\tau=0$ a.s. for all $m$ and for all $1 \leq i \leq n$, we obtain
\begin{align}
    \mathbb{E}\left[|\Delta V^{i,n}_\tau|^p\right]=\mathbb{E}\left[|\Delta V^{i,n}_\tau-\Delta V^{(m),i,n}_\tau|^p\right] \leq 2^p \sup_{\tau \in \mathcal{T}_0}\mathbb{E}\left[|V^{i,n}_\tau-V^{(m),i,n}_\tau|^p\right],
\end{align}
which implies that $\Delta V^{i,n}_\tau=0$ a.s. for all $1 \le i \le n$ and consequently, $h_1(t,V^{i}_t,L_n[\textbf{V}^n_t])$ and $-h_2(t,V^{i}_t,L_n[\textbf{V}^n_t])$ are left upper semicontinuous along stopping times. By Theorem 3.7 (ii) in \cite{dqs16}, for each $1 \le i \le n$, we get that $(\tau_{S}^{\star,i,n}, \sigma_{S}^{\star,i,n})$ given by $\eqref{S}$ is a $S$-saddle point.
\end{proof}

\section{Propagation of chaos}\label{chaos}

This section is concerned with the convergence of the sequence of value functions $V^{i,n}$ of the interacting Dynkin games to i.i.d. copies of the value function $V$ of the mean-field Dynkin game which, by the results from the previous section, consists in showing the convergence of the component $Y^{i,n}$ of the solution of the interacting system of doubly reflected BSDEs \eqref{BSDEParticle} to i.i.d. copies of the component $Y$ of the mean-field doubly Reflected BSDE with driver $f\circ Y$ and terminal condition $h_1(\tau,Y_\tau,\P_{Y_s|s=\tau})\ind_{\{\tau\le \sigma<T\}}+h_2(\sigma,Y_\sigma,\P_{Y_s|s=\sigma})\ind_{\{\sigma<\tau\}}+\xi^{i}\ind_{\{\tau \wedge \sigma=T\}}$. These convergence results yield the propagation of chaos result, as it will be explained below.

\medskip
To establish the propagation of chaos property of the particle system \eqref{BSDEParticle}, we make the following additional assumptions.

\begin{Assumption}\label{Assump:chaos}
\begin{itemize}
    \item [(i)] The sequence  $ \xi^n=(\xi^{1,n},\xi^{2,n},\ldots,\xi^{n,n})$ is exchangeable i.e. the sequence of probability laws $\mu^n$ of $\xi^n$ on $\R^n$ is symmetric.
    
   \item[(ii)] For each $i\ge 1$, $\xi^{i,n}$ converges in $L^p$ to $ \xi^i$, i.e. $$\lim_{n\to\infty}\E[|\xi^{i,n}-\xi^i|^p]=0,$$
        
        where the random variables $\xi^{i}\in L^p(\Fc^i_T)$ are independent and equally distributed (iid) with probability law $\mu$.
    \item[(iii)] The component $(Y_t)$ of the unique solution of the mean-field doubly reflected BSDE  \eqref{BSDE1} has jumps only at totally inaccessible stopping times.
\end{itemize}
    \end{Assumption}

For $m\ge 1$, introduce the Polish spaces
$\mathbb{H}^{2,m}:=L^2([0,T];\R^m)$ and $\mathbb{H}_{\nu}^{2,m}$, the space of measurable functions $\ell: \,[0,T]\times \R^*\longrightarrow \R^n$ such that $\int_0^T\int_{\R^*}\sum_{j=1}^m |\ell^j(t,u)|^2\nu(du)dt< \infty$.     
    
\medskip
In the following proposition, we show that the \textit{exchangeability property} transfers from the terminal conditions to the associated solution processes.
    \begin{Proposition}[Exchangeability property]\label{exchange}
Assume the sequence  $ \xi^n=(\xi^{1,n},\xi^{2,n},\ldots,\xi^{n,n})$ is exchangeable i.e. the sequence of probability laws $\mu^n$ of $\xi^n$ on $\R^n$ is symmetric. Then the processes $(Y^{i,n},Z^{i,n},U^{i,n},K^{1,i,n},K^{2,i,n}),\,i=1,\ldots, n,$ solutions of  the systems \eqref{BSDEParticle} are exchangeable.
\end{Proposition}

The proof is similar to that of \cite[Proposition 4.1]{ddz21}. We omit it.

Consider the product space 
$$
G:=\mathbb{D}\times \mathbb{H}^{2,n}\times \mathbb{H}_{\nu}^{2,n}\times \mathbb{D} 
$$
endowed with the product metric 
$$
\delta(\theta,\theta^{\prime}):=\left(d^o(y,y^{\prime})^p+\|z-z^{\prime}\|^p_{\mathbb{H}^{2,n}}+\|u-u^{\prime}\|^p_{\mathbb{H}_{\nu}^{2,n}}+d^o(k,k^{\prime})^p\right)^{\frac{1}{p}}.
$$
where, $\theta:=(y,z,u,k)$ and $\theta^{\prime}:=(y^{\prime},z^{\prime},u^{\prime},k)$.

We define the Wasserstein metric on $\mathcal{P}_p(G)$ by 
\begin{equation}\label{W-simultan} 
D_G(P,Q)=\inf\left\{\left(\int_{G\times G}  \delta(\theta,\theta^{\prime})^p R(d\theta,d\theta^{\prime})\right)^{1/p}\right\},
\end{equation}
over $R\in\mathcal{P}(G\times G)$ with marginals $P$ and $Q$. Since $(G,\delta)$ is a Polish space, $(\mathcal{P}_p(G), D_G)$ is a Polish space and induces the topology of weak convergence.

\medskip
Let $(Y^i,Z^i,U^i,K^{1,i},K^{2,i}), i=1,\ldots, n$, with independent terminal values $Y^i_T=\xi^i,i=1,\ldots, n$, be independent copies of $(Y,Z,U,K^1, K^2)$, the solution of \eqref{BSDE1}. 
More precisely, for each $i=1,\ldots,n$, $(Y^i,Z^i,U^i,K^{1,i},K^{2,i})$ is the unique solution of the reflected MF-BSDE
\begin{align}\label{BSDEParticle-Theta}
    \begin{cases} \quad
    Y^{i}_t = \xi^{i} +\int_t^T f(s,Y^{i}_{s},Z^{i}_{s},U^{i}_{s},\mathbb{P}_{Y_s^{i}})ds + K^{1,i}_T - K^{1,i}_t- K^{2,i}_T + K^{2,i}_t\\ \qquad\qquad\qquad \qquad -\int_t^T Z^{i}_s dB^i_s -  \int_t^T \int_{\R^\star}  U^{i}_s(e) \tilde N^i(ds,de), \\ \quad 
    h_2(t,Y^{i}_{t},\mathbb{P}_{Y_t^{i}}) \geq Y^{i}_{t} \geq h_1(t,Y^{i}_{t},\mathbb{P}_{Y_t^{i}}),  \quad \forall t \in [0,T]; \,\, Y_T^{i}=\xi^{i}, \\ \quad
    \int_0^T (Y^{i}_{t^-}-h_1(t^-,Y^{i}_{t^-},\mathbb{P}_{Y_{t^-}^{i}}))dK^{1,i}_t = 0, \\ \quad
    \int_0^T (Y^{i}_{t^-}-h_2(t^-,Y^{i}_{t^-},\mathbb{P}_{Y_{t^-}^{i}}))dK^{2,i}_t = 0,\\ \quad
    dK^{1,i}_t \perp dK^{2,i}_t.
    \end{cases}
\end{align}
In the sequel, we denote $(f \circ Y^{i})(t,y,z,u):=f(t,y,z, u,\mathbb{P}_{Y^{i}_t})$.

Introduce the notation
\begin{equation}\label{W}
W:=K^1-K^2,\quad W^i:=K^{1,i}-K^{2,i},\quad W^{i,n}:=K^{1,i,n}-K^{2,i,n},
\end{equation}

and consider the processes
\begin{equation}\label{Theta}
\Theta:=(Y,Z,U,W),\quad \Theta^{i}:=(Y^i,Z^i,U^i,W^i),\quad \Theta^{i,n}:=(Y^{i,n},Z^{i,n},U^{i,n},W^{i,n}).
\end{equation}

\noindent For any fixed $1\leq k\leq n$,  let
\begin{equation}\label{P-Theta}
\P^{k,n}:=\text{Law}\,(\Theta^{1,n},\Theta^{2,n},\ldots,\Theta^{k,n}),\quad \P_{\Theta}^{\otimes k}:=\text{Law}\,(\Theta^1,\Theta^2,\ldots,\Theta^k).
\end{equation}
be the joint probability laws of the processes $(\Theta^{1,n},\Theta^{2,n},\ldots,\Theta^{k,n})$ and $(\Theta^1,\Theta^2,\ldots,\Theta^k)$, respectively.

From the definition of the distance $D_G$, we obtain the following inequality. 
\begin{align}\label{chaos-0}
D_G^p(\P^{k,n},\P_{\Theta}^{\otimes k})&\le 
k\sup_{i\le k}\left(\|Y^{i,n}-Y^i\|^p_{\sp}+\|Z^{i,n}-Z^i{\bf e}_i\|^p_{\mathcal{H}^{p,n}}  \right. \nonumber \\  &\left. \qquad \qquad\qquad +\|U^{i,n}-U^i{\bf e}_i\|^p_{\mathcal{H}^{p,n}_{\nu}} +\|W^{i,n}-W^{i}\|^p_{\sp}\right),
\end{align}
where for each $i=1,\ldots, n$,  ${\bf e}_i:=(0,\ldots,0,\underbrace{1}_{i},0,\ldots,0)$.

Before we state and prove the propagation of chaos result, we give here the following convergence result of the empirical laws of i.i.d. copies $Y^1,Y^2,\ldots,Y^n$ solutions of \eqref{BSDEParticle-Theta} and convergence of the particle system \eqref{BSDEParticle} to the solution to \eqref{BSDE1}.

\begin{Theorem}[Law of Large Numbers]\label{LLN-2} Let $Y^1,Y^2,\ldots,Y^n$ with terminal values $Y^i_T=\xi^i$ be independent copies of  the solution $Y$ of \eqref{BSDE1}. Then, we have
\begin{equation}\label{GC-2}
\lim_{n\to\infty}\E\left[\underset{0\le t\le T}\sup\, \mathcal{W}_p^p(L_n[\textbf{Y}_t],\P_{Y_t})\right]=0.
\end{equation}
\end{Theorem}
The proof is similar to that of \cite[Theorem 4.1]{ddz21}, so we omit it.

We now provide the following convergence result  for the solution $Y^{i,n}$ of \eqref{BSDEParticle}.

\medskip
\begin{Proposition}[Convergence of the $Y^{i,n}$'s]\label{chaos-Y-1}
Assume that for some $p\ge 2$,  $\gamma_1$, $\gamma_2$, $\kappa_1$ and $\kappa_2$ satisfy 
\begin{align}\label{smallnessCond-chaos-1}
  2^{\frac{p}{2}-1} \textcolor{black}{7^{p-1}}(\gamma_1^p+\gamma_2^p+\kappa_1^p +\kappa_2^p)<1.
\end{align}
Then, under Assumptions \ref{generalAssump}, \ref{Assump-PS} and \ref{Assump:chaos}, we have
\begin{equation}\label{chaos-Y-1-1}
\underset{n\to\infty}\lim \,\underset{0\le t\le T}{\sup}\,E[|Y^{i,n}_t-Y^i_t|^p]=0.
\end{equation}
In particular,
\begin{equation}\label{chaos-Y-1-2}
\underset{n\to\infty}\lim \, \|Y^{i,n}-Y^i\|_{\mathcal{H}^{p,1}}=0.
\end{equation}
\end{Proposition}
\begin{proof}
Given $t\in[0,T]$, let $\vartheta \in \mathcal{T}^n_t$. By the estimates on BSDEs in Proposition \ref{estimates} (see the appendix below), we have
\begin{equation}\label{Y-n-estimate}\begin{array} {lll}
|Y^{i,n}_\vartheta-Y^{i}_\vartheta|^{p} \\
\le \underset{\tau\in\mathcal{T}^n_\vartheta}{\esssup\,}\underset{\sigma \in\mathcal{T}^n_\vartheta}{\esssup\,}\left|\mathcal{E}_{\vartheta,\tau \wedge \sigma}^{\textbf{f}^i \circ \textbf{Y}^{n}}[h_1(\tau,Y^{i,n}_\tau,L_{n}[\textbf{Y}^n_\tau])\ind_{\{\tau \leq  \sigma<T\}}+h_2(\sigma,Y^{i,n}_\sigma,L_{n}[\textbf{Y}^n_\sigma])\ind_{\{\sigma < \tau\}}+\xi^{i,n}\ind_{\{\tau \wedge \sigma=T\}}] \right. \\ \left. \qquad\qquad -\mathcal{E}_{\vartheta,\tau \wedge \sigma}^{\textbf{f}^i \circ Y^i }[h_1(\tau,Y^i_\tau,\P_{Y_s|s=\tau})\ind_{\{\tau \le \sigma<T\}}+h_2(\sigma,Y^i_\sigma,\P_{Y_s|s=\sigma})\ind_{\{\sigma < \tau\}}+\xi^i \ind_{\{\tau \wedge \sigma=T\}}] \right|^{p} \\
\le \underset{\tau\in\mathcal{T}^n_\vartheta}{\esssup\,} \underset{\sigma \in\mathcal{T}^n_\vartheta}{\esssup\,}2^{\frac{p}{2}-1}\E\left[\eta^p(\int_\vartheta^{\tau \wedge \sigma} e^{2 \beta (s-\vartheta)} |(\textbf{f}^i \circ \textbf{Y}^{n}) (s,\widehat{Y}^{i,\tau, \sigma}_s,\widehat{Z}^{i,\tau, \sigma}_s,\widehat{U}^{i,\tau, \sigma}_s) \right. \\ \left. \qquad\quad -(\textbf{f}^i \circ Y^i) (s,\widehat{Y}^{i,\tau, \sigma}_s, \widehat{Z}^{i,\tau, \sigma}_s,\widehat{U}^{i,\tau, \sigma}_s)|^2 ds)^{\frac{p}{2}} \right. \\ \left. \qquad\qquad +\left(e^{ \beta(\tau \wedge \sigma-\vartheta)}\left[|h_1(\tau,Y^{i,n}_\tau,L_{n}[\textbf{Y}^n_\tau])-h_1(\tau,Y^i_\tau, \P_{Y_s|s=\tau})|\ind_{\{\tau \leq  \sigma<T\}} \right. \right. \right. \nonumber\\ \left. \left. \left. \qquad\qquad\qquad + |h_2(\sigma,Y^{i,n}_\sigma,L_{n}[\textbf{Y}^n_\sigma])-h_2(\sigma,Y^i_\sigma, \P_{Y_s|s=\sigma})|\ind_{\{\sigma < \tau\}}\right] + e^{\beta(T-\vartheta)}|\xi^{i,n}-\xi^i|\ind_{\{\tau \wedge \sigma=T\}}\right)^p
|\mathcal{F}_\vartheta\right]\\
\le \underset{\tau\in\mathcal{T}^n_\vartheta}{\esssup\,} \underset{\sigma \in\mathcal{T}^n_\vartheta}{\esssup\,} 2^{\frac{p}{2}-1}\E\left[\int_\vartheta^{\tau \wedge \sigma}T^{\frac{p-2}{2}} e^{p \beta (s-\vartheta)}\eta^p C_f^{p}\mathcal{W}^p_p(L_{n}[\textbf{Y}_s^n],\P_{Y_s})ds +\left( \gamma_1 e^{ \beta(\tau-\vartheta)}|Y^{i,n}_{\tau}-Y^{i}_{\tau}| \right. \right. \\ \left.\left. \qquad\qquad\qquad +\gamma_2 e^{ \beta(\tau-\vartheta)} \mathcal{W}_p(L_{n}[\textbf{Y}_{\tau}^n],\P_{Y_s|s=\tau})+\kappa_1 e^{ \beta(\sigma-\vartheta)}|Y^{i,n}_{\sigma}-Y^{i}_{\sigma}| \right. \right. \\ \left.\left. \qquad\qquad\qquad\qquad\qquad +\kappa_2 e^{ \beta(\sigma-\vartheta)} \mathcal{W}_p(L_{n}[\textbf{Y}_{\sigma}^n],\P_{Y_s|s=\sigma})  +e^{ \beta(T-\vartheta)}|\xi^{i,n}-\xi^i|\ind_{\{\tau \wedge \sigma=T\}}
\right)^p|\mathcal{F}_\vartheta\right],
 \end{array}
\end{equation}
with $\eta$, $\beta>0$ such that  $\eta \leq \frac{1}{C_f^2}$ and  $\beta \geq 2 C_f+\frac{3}{\eta}$, where $(\widehat{Y}^{i,\tau,\sigma},\widehat{Z}^{i,\tau, \sigma},\widehat{U}^{i,\tau, \sigma})$ is the solution of the BSDE associated with driver $\textbf{f}^{i}\circ \Ybf^n$, terminal time $\tau \wedge \sigma$ and terminal condition $h_1(\tau,Y^{i,n}_\tau,L_{n}[\Ybf^n_s]_{s=\tau})\ind_{\{\tau \leq  \sigma<T\}}+h_2(\sigma,Y^{i,n}_\sigma,L_{n}[\Ybf^n_s]_{s=\sigma})\ind_{\{\sigma < \tau\}}+\xi^{i,n}\ind_{\{\tau \wedge \sigma=T\}}$.

Therefore, we have
\begin{equation*}
e^{p\beta \vartheta}|Y^{i,n}_\vartheta-Y^{i}_\vartheta|^{p}\le \underset{\tau\in\mathcal{T}^n_\vartheta}{\esssup\,}\E[G^{i,n}_{\vartheta,\tau}|\mathcal{F}_\vartheta]+\underset{\tau\in\mathcal{T}^n_\vartheta}{\esssup\,}\E[H^{i,n}_{\vartheta,\tau}|\mathcal{F}_\vartheta],
\end{equation*}
where
\begin{equation*}\begin{array} {lll}
G^{i,n}_{\vartheta,\tau}:=2^{\frac{p}{2}-1} \left\{\int_\vartheta^{\tau}T^{\frac{p-2}{2}}2^{p-1}\eta^p C_f^{p} \left(\frac{1}{n}\sum_{j=1}^ne^{p \beta s}|Y^{j,n}_s-Y^j_s|^p\right) ds \right.
\\ \left. \qquad\qquad\qquad +\textcolor{black}{7^{p-1}}(\gamma_1^p+\gamma_2^p) \left(e^{p\beta\tau}|Y^{i,n}_{\tau}-Y^{i}_{\tau}|^p+
\frac{1}{n}\sum_{j=1}^ne^{p\beta\tau}|Y^{j,n}_{\tau}-Y^j_{\tau}|^p\right)  \right.
\\ \left. \qquad\qquad\qquad + \left(\textcolor{black}{7^{p-1}}\gamma_2^p+2^{p-1}T\eta^p T^{\frac{p-2}{2}} C_f^{p}\right)\underset{0\le t\le T}{\sup}\, e^{p \beta s}\mathcal{W}^p_p(L_{n}[\textbf{Y}_s],\P_{Y_s})+ \textcolor{black}{7^{p-1}}e^{\beta T}|\xi^{i,n}-\xi^i|^p\ind_{\{\tau=T\}}. \right\}
\end{array}
\end{equation*}
and
\begin{equation*}\begin{array} {lll}
H^{i,n}_{\vartheta,\tau}:=2^{\frac{p}{2}-1} \textcolor{black}{7^{p-1}}(\kappa_1^p+\kappa_2^p) \left(e^{p\beta\tau}|Y^{i,n}_{\tau}-Y^{i}_{\tau}|^p+
\frac{1}{n}\sum_{j=1}^ne^{p\beta\tau}|Y^{j,n}_{\tau}-Y^j_{\tau}|^p\right) 
\\  \qquad\qquad\qquad + 2^{\frac{p}{2}-1} \textcolor{black}{7^{p-1}}\kappa_2^p \underset{0\le t\le T}{\sup}\, e^{p \beta s}\mathcal{W}^p_p(L_{n}[\textbf{Y}_s],\P_{Y_s}).
\end{array}
\end{equation*}
Setting
$V^{n,p}_t:=\frac{1}{n}\sum_{j=1}^ne^{p \beta t}|Y^{j,n}_t-Y^j_t|^p$
and 
$$
\Gamma_{n,p}:=2^{\frac{p}{2}-1} \left(\textcolor{black}{7^{p-1}}\gamma_2^p+\textcolor{black}{7^{p-1}}\kappa_2^p+2^{p-1}T\eta^p T^{\frac{p-2}{2}} C_f^{p}\right)\underset{0\le t\le T}{\sup}\, e^{p \beta s}\mathcal{W}^p_p(L_{n}[\textbf{Y}_s],\P_{Y_s})+ 2^{\frac{p}{2}-1} \textcolor{black}{7^{p-1}}e^{\beta T}|\xi^{i,n}-\xi^i|^p.
$$
we obtain 
\begin{align}\label{V-n-p}
V^{n,p}_\vartheta &\le \underset{\tau\in\mathcal{T}^n_\vartheta}{\esssup\,}\E[\int_\vartheta^{\tau} 2^{\frac{p}{2}-1} 2^{p-1}T^{\frac{p-2}{2}}\eta^p C_f^{p} V^{n,p}_sds+2^{\frac{p}{2}-1} \textcolor{black}{7^{p-1}}(\gamma_1^p+\gamma_2^p)V^{n,p}_{\tau}+\Gamma_{n,p}|\mathcal{F}_\vartheta] \nonumber \\ &+\underset{\tau\in\mathcal{T}^n_\vartheta}{\esssup\,}\E[2^{\frac{p}{2}-1} \textcolor{black}{7^{p-1}}(\kappa_1^p+\kappa_2^p)V^{n,p}_{\tau}|\mathcal{F}_\vartheta].
\end{align}
Therefore, we get 
$$
\E[V^{n,p}_\vartheta]\le 2^{\frac{p}{2}-1} \textcolor{black}{7^{p-1}}(\gamma_1^p+\gamma_2^p+\kappa_1^p+\kappa_2^p)\underset{\tau\in\mathcal{T}^n_\vartheta}{\sup}\,\E[V^{n,p}_{\tau}]+\E[\int_\vartheta^T 2^{p-1}2^{\frac{p}{2}-1} T^{\frac{p-2}{2}}\eta^p C_f^{p} V^{n,p}_sds+\Gamma_{n,p}].
$$
Since $\mathcal{T}^n_\vartheta \subset \mathcal{T}^n_t$ and by arbitrariness of $\vartheta \in \mathcal{T}^n_t$, we obtain
$$
\lambda \E[V^{n,p}_{t}] \leq \lambda \underset{ \vartheta \in \mathcal{T}_t^n}{\sup}\,\E[V^{n,p}_{\vartheta}]\le \E[\int_t^T 2^{p-1} 2^{\frac{p}{2}-1} \eta^p T^{\frac{p-2}{2}} C_f^{p} V^{n,p}_sds+\Gamma_{n,p}],
$$
where $\lambda:=1-2^{\frac{p}{2}-1} \textcolor{black}{7^{p-1}}(\gamma_1^p+\gamma_2^p+\kappa_1^p+\kappa_2^p) >0$ by the assumption \eqref{smallnessCond-chaos-1}. By \textcolor{black}{Gronwall's inequality}, we have
$$
\underset{0\le t\le T}{\sup}\,\E[V^{n,p}_{t}]\le \frac{ e^{K_p}}{\lambda}\E[\Gamma_{n,p}]
$$
where $K_p:=\frac{1}{\lambda}2^{p-1}2^{\frac{p}{2}-1} T^{\frac{p-2}{2}}\eta^p C_f^{p}T$. But, in view of the exchangeability of the processes $(Y^{i,n},Y^i),i=1,\ldots,n$ (see Proposition \ref{exchange}), we have, $\E[V^{n,p}_{t}]=\E[e^{p\beta t}|Y^{i,n}_t-Y^i_t|^p]$. Thus,
$$
\underset{0\le t\le T}{\sup}\,\E[e^{p\beta t}|Y^{i,n}_t-Y^i_t|^p]\le \frac{ e^{K_p}}{\lambda}\E[\Gamma_{n,p}] \to 0
$$
as $n\to\infty$,
in view of Theorem \ref{LLN-2} and Assumption \ref{Assump:chaos}, as required.  \qed
\end{proof}

We now derive the following propagation of chaos result.

\begin{Theorem}[Propagation of chaos of the $Y^{i,n}$'s]\label{prop-y} 
Under the assumptions of Proposition \ref{chaos-Y-1}, the solution $Y^{i,n}$ of the particle system \eqref{BSDEParticle} satisfies the propagation of chaos property, i.e. for any fixed positive integer $k$, 
$$
\underset{n\to\infty}{\lim} \text{Law}\,(Y^{1,n},Y^{2,n},\ldots,Y^{k,n})=\text{Law}\,(Y^1,Y^2,\ldots,Y^k).
$$
\end{Theorem}

\begin{proof}
 Set $P^{k,n}:=\text{law}\,(Y^{1,n},Y^{2,n},\ldots,Y^{k,n})$ and $P^{\otimes k}:=\text{Law}\,(Y^1,Y^2,\ldots,Y^k)$.
Consider the Wasserstein metric on $\mathcal{P}_2(\mathbb{H}^2)$ defined by 
\begin{equation}\label{W-Y} 
D_{\mathbb{H}^2}(P,Q)=\inf\left\{\left(\int_{\mathbb{H}^2\times \mathbb{H}^2}  \|y-y^{\prime}\|_{\mathbb{H}^2}^2 R(dy,dy^{\prime})\right)^{1/2}\right\},
\end{equation}
over $R\in\mathcal{P}(\mathbb{H}^2\times \mathbb{H}^2)$ with marginals $P$ and $Q$. Note that, since $p \geq 2$, it is enough to show for $D_{\mathbb{H}^2}(P,Q)$. Since $\mathbb{H}^2$ is a Polish space, $(\mathcal{P}_2(\mathbb{H}^2), D_{\mathbb{H}^2})$ is a Polish space and induces the topology of weak convergence. Thus, we obtain the propagation of chaos property for the $Y^{i,n}$'s if we can show that 
$\underset{n\to\infty}{\lim}D_{\mathbb{H}^2}(P^{k,n},P^{\otimes k})=0$. But, this follows from the fact that 
$$
D^2_{\mathbb{H}^2}(P^{k,n},P^{\otimes k})\le k\sup_{i\le k}\|Y^{i,n}-Y^i\|_{\mathcal{H}^{2,1}}^2
$$
and \eqref{chaos-Y-1-2} for $p=2$. \qed

\end{proof}

\medskip

\textcolor{black}{In the next proposition we show a convergence result for $p>2$, of the whole solution $(Y^{i,n},Z^{i,n},U^{i,n},W^{i,n})$ of the system \eqref{BSDEParticle}}.

\begin{Proposition}\label{chaos-1} Assume that, for some $p>2$, Assume that  $\gamma_1$, $\gamma_2$, $\kappa_1$ and $\kappa_2$ satisfy 
\begin{align}\label{smallnessCond-chaos}
  2^{9p/2-3}(\gamma_1^p+\gamma_2^p+\kappa_1^p+\kappa_2^p)<\left(\frac{p-\kappa}{2p}\right)^{p/\kappa}
\end{align}
for some $\kappa\in [2,p)$.

Then, under Assumptions \ref{generalAssump}, \ref{Assump-PS} and \ref{Assump:chaos}, we have
\begin{equation}\label{chaos-1-1}\begin{array}{ll}
\underset{n\to\infty}\lim \,\left(\|Y^{i,n}-Y^i\|_{\sp}+\|Z^{i,n}-Z^i{\bf e}_i\|_{\mathcal{H}^{p,n}}+ \|U^{i,n}-U^i{\bf e}_i\|_{\mathcal{H}^{p,n}_\nu}+ \|W^{i,n}-W^i\|_{\sp}\right)=0.
\end{array}
\end{equation}
Here, ${\bf e}_1,\ldots,{\bf e}_n$ denote unit vectors in $\R^n$. 
\end{Proposition}

\begin{proof} 
\noindent \uline{Step 1.} In view of \eqref{Y-n-estimate}, for any $\kappa \ge 2$ and any $t\leq T$, we have
\begin{equation*}\begin{array} {lll}
|Y^{i,n}_t-Y^{i}_t|^{\kappa} 
\le \underset{\tau\in\mathcal{T}^n_t}{\esssup\,} \underset{\sigma\in\mathcal{T}^n_t}{\esssup\,}2^{\frac{\kappa}{2}-1}\E\left[\int_t^{\tau \wedge \sigma}e^{\kappa \beta (s-t)}\eta^{\kappa} T^{\frac{\kappa-2}{\kappa}}C_f^{\kappa}\mathcal{W}^{\kappa}_{p}(L_{n}[\textbf{Y}_s^n],\P_{Y_s})ds \right. \\ + \left. \left( 
\gamma_1 e^{ \beta(\tau-t)}|Y^{i,n}_{\tau}-Y^{i}_{\tau}|+\gamma_2 e^{ \beta(\tau-t)} \mathcal{W}_{p}(L_{n}[\textbf{Y}_{\tau}^n],\P_{Y_s|s=\tau}) \right. \right. \\ \left. \left. +\kappa_1 e^{ \beta(\sigma-t)}|Y^{i,n}_{\sigma}-Y^{i}_{\sigma}|+\kappa_2 e^{ \beta(\sigma-t)} \mathcal{W}_{p}(L_{n}[\textbf{Y}_{\sigma}^n],\P_{Y_s|s=\sigma})  +e^{ \beta(T-t)}|\xi^{i,n}-\xi^i|\ind_{\{\tau \wedge \sigma=T\}}
\right)^{\kappa}|\mathcal{F}_t\right],
 \end{array}
\end{equation*}
where $\eta$, $\beta>0$ such that  $\eta \leq \frac{1}{C_f^2}$ and  $\beta \geq 2 C_f+\frac{3}{\eta}$.

Therefore, for any $p>\kappa \ge 2$, we have
\begin{equation*}
e^{p\beta t}|Y^{i,n}_t-Y^{i}_t|^{p}\le \E[\mathcal{G}^{i,n}_{T}|\mathcal{F}_t]^{p/\kappa},
\end{equation*}
where
$$
\begin{array}{lll}
\mathcal{G}_T^{i,n}:=2^{\frac{\kappa}{2}-1} \left\{\int_0^TT^{\frac{\kappa-2}{\kappa}}e^{\kappa \beta s}\eta^{\kappa} C_f^{\kappa}\mathcal{W}^{\kappa}_{p}(L_{n}[\textbf{Y}_s^n],\P_{Y_s})ds \right. \\ \left. \qquad\qquad\qquad +\left( 
(\gamma_1+\kappa_1) \underset{0\le s\le T}{\sup}\,e^{\beta s}|Y^{i,n}_{s}-Y^{i}_{s}| +(\gamma_2+\kappa_2) \underset{0\le s\le T}{\sup}\,e^{ \beta s} \mathcal{W}_{p}(L_{n}[\textbf{Y}_{s}^n],\P_{Y_s}) +e^{ \beta T}|\xi^{i,n}-\xi^i|\right)^{\kappa}\right\}.
\end{array}
$$
Thus, by Doob's inequality, we get
\begin{equation}\label{Y-n-Doob}
\E[\underset{0\le t\le T}{\sup}\,e^{p\beta t}|Y^{i,n}_t-Y^{i}_t|^{p}]\le \left(\frac{p}{p-\kappa}\right)^{p/\kappa}\E\left[\left(\mathcal{G}_T^{i,n}\right)^{p/\kappa}\right].
\end{equation}

Therefore, we have
$$\begin{array}{lll}
2^{1-\frac{p}{\kappa}}\left(\mathcal{G}_T^{i,n}\right)^{p/\kappa}\le C_1\int_0^T  \underset{0\le t\le s}{\sup}\,e^{p\beta t}\mathcal{W}_p^p(L_{n}[\textbf{Y}_{t}^n],L_{n}[\textbf{Y}_{t}])ds \\ \qquad\qquad\qquad\qquad +2^{3p/2-2}\left((\gamma_1+\kappa_1)\underset{0\le t\le T}{\sup}\,e^{\beta t}|Y^{i,n}_t-Y^{i}_t|+(\gamma_2+\kappa_2)\underset{0\le t\le T}{\sup}\,e^{\beta t}\mathcal{W}_{p}(L_{n}[\textbf{Y}_{s}^n],L_{n}[\textbf{Y}_{s}])\right)^p +\Lambda_n
\end{array}
$$
where $C_1:=2^{3p/2-2}T^{2\frac{p}{\kappa}-1}\eta^pC_f^p$ and
$$
\Lambda_n:=C_1 T\underset{0\le s\le T}{\sup}\,e^{ p\beta s} \mathcal{W}_{p}(L_{n}[\textbf{Y}_{s}],\P_{Y_s})+2^{3p/2-2}\left(
(\gamma_2+\kappa_2) \underset{0\le s\le T}{\sup}\,e^{ \beta s} \mathcal{W}_{p}(L_{n}[\textbf{Y}_{s}],\P_{Y_s}) +e^{ \beta T}|\xi^{i,n}-\xi^i|\right)^{p}.
$$
\noindent But, in view of the exchangeability of the processes $(Y^{i,n},Y^i),i=1,\ldots,n$ (see Proposition \ref{exchange}), for each $s \in [0,T]$ we have, 
$$
\E\left[\underset{0\le t\le s}{\sup}e^{\beta p t}\mathcal{W}_p^p(L_n[\textbf{Y}_t^n],L_n[\textbf{Y}_t])\right]\leq \E\left[\underset{0\le t\le s}{\sup}e^{\beta p t}|Y^{i,n}_t-Y^{i}_t|^p\right],\quad i=1,\ldots,n,
$$
and so
$$\begin{array}{lll}
2^{1-\frac{p}{\kappa}}\E[\left(\mathcal{G}_T^{i,n}\right)^{p/\kappa}]\le C_1\E\left[\int_0^T \underset{0\le t\le s}{\sup}\,e^{p\beta t}\mathcal{W}_p^p(L_{n}[\textbf{Y}_{t}^n],L_{n}[\textbf{Y}_{t}])ds\right]+\E\left[\Lambda_n\right]\\ \qquad\qquad\qquad\qquad\qquad\qquad
+2^{9p/2-3}(\gamma_1^p+\gamma_2^p+\kappa_1^p+\kappa_2^p)\E[\underset{0\le t\le T}{\sup}\,e^{p\beta t}|Y^{i,n}_t-Y^{i}_t|^{p}].
\end{array}
$$
Therefore, \eqref{Y-n-Doob} becomes
\begin{equation*}
\begin{array} {lll}
\mu\E\left[\underset{0\le t\le
T}{\sup}e^{p\beta t}|Y^{i,n}_t-Y^{i}_t|^p\right] \le C_1\E\left[\int_0^T \underset{0\le t\le s}{\sup}e^{\beta p t}|Y^{i,n}_t-Y^{i}_t|^pds\right]+\E\left[\Lambda_n\right].
\end{array}
\end{equation*}
where $\mu:=2^{1-\frac{p}{\kappa}}\left(\frac{p}{p-\kappa}\right)^{-p/\kappa}-2^{9p/2-3} (\gamma_1^p+\gamma_2^p+\kappa_1^p+\kappa_2^p)$.
Using the condition \eqref{smallnessCond-chaos}, to see that
$\mu>0$, and Gronwall's inequality, we finally  obtain
\begin{equation*}
\E\left[\underset{0\le t\le
T}{\sup}e^{p\beta t}|Y^{i,n}_t-Y^{i}_t|^p\right] \le e^{\frac{C_1}{\mu}T}\E\left[\Lambda_n\right].
\end{equation*}
Next, by \eqref{GC-2} together with Assumption \eqref{Assump:chaos} (iii) we have
$$
\underset{n\to\infty}\lim\, \E\left[\Lambda_n\right]= 0,
$$
which yields the desired result.

\medskip
\noindent \uline{Step 2.} We now prove that $\underset{n\to\infty}{\lim}\|Z^{i,n}-Z^i{\bf e}_i\|_{\mathcal{H}^{p,n}}=0,    \,\underset{n\to\infty}{\lim}\|U^{i,n}-U^i{\bf e}_i\|_{\mathcal{H}^{p,n}_\nu}=0$ and $\underset{n\to\infty}{\lim}\|W^{i,n}-W^i\|_{\mathcal{S}^{p}}=0$. We start by showing that $\underset{n\to\infty}{\lim}\|Z^{i,n}-Z^i{\bf e}_i\|_{\mathcal{H}^{p,n}}=0$ and    $\underset{n\to\infty}{\lim}\|U^{i,n}-U^i{\bf e}_i\|_{\mathcal{H}^{p,n}_\nu}=0$. For $s \in [0,T]$, denote $\delta Y^{i,n}_s:= Y_s^{i,n}-Y_s^{i}$, $\delta Z_s^{i,n}:= Z_s^{i,n}-Z_s^{i}{\bf e}_i$, $\delta U^{i,n}_s:= U_s^{i,n}-U_s^{i}{\bf e}_i$, $\delta K_s^{1,i,n}:= K_s^{1,i,n}-K_s^{1,i}$, $\delta K_s^{2,i,n}:= K_s^{2,i,n}-K_s^{2,i}$, $\delta f^{i,n}_s:=f(s,Y_s^{i,n}, Z_s^{i,i,n}, U_s^{i,i,n}, L_{n}[\textbf{Y}_s^n])-f(s,Y_s^{i}, Z_s^{i}, U_s^{i}, \mathbb{P}_{Y^{i}_s})$, $\delta \xi^{i,n}:=\xi^{i,n}-\xi^{i}$, $\delta h^{1,i,n}_s:=h^1(s,Y_s^{i,n},L_{n}[\textbf{Y}^n_s])-h^1(s,Y_s^{i}, \mathbb{P}_{Y^{i}_s})$ and $\delta h^{2,i,n}_s:=h^2(s,Y_s^{i,n},L_{n}[\textbf{Y}^n_s])-h^2(s,Y_s^{i}, \mathbb{P}_{Y^{i}_s})$. By applying It\^{o}'s formula to $|\delta Y^{i,n}_t|^2$, we obtain
\begin{align*}
    &|\delta Y^{i,n}_t|^2+\int_t^T |\delta Z^{i,n}_s|^2ds+\int_t^T\int_{\R^*}\sum_{j=1}^n|\delta U^{i,j,n}_s(e)|^2 N^j(ds,de)+\sum_{t<s\leq T}|\Delta W_s^{i,n}-\Delta W_s^{i}|^2=  |\delta \xi^{i,n}|^2 \nonumber \\
    &  + 2\int_t^T \delta Y_s^{i,n}\delta f_s^{i,n}ds - 2 \int_t^T \delta Y_s^{i,n} \sum_{j=1}^n\delta Z_s^{i,j,n} dB^j_s - 2 \int_t^T \int_{\R^*} \delta Y_{s^-}^{i,n} \sum_{j=1}^n\delta U_s^{i,j,n}(e) \Tilde{N}^j(ds,de) \nonumber \\ & +2\int_t^T \delta Y_{s^-}^{i,n} d(\delta K_s^{1,i,n})-2\int_t^T \delta Y_{s^-}^{i,n} d(\delta K_s^{2,i,n}).
\end{align*}
By standard estimates, from the assumptions on the driver $f$, we get, for all $\varepsilon>0$,
\begin{equation*}\begin{array}{lll}
    \int_t^T\delta Y_s^{i,n}\delta f_s^{i,n}ds \leq \int_t^T C_f|\delta Y_s^{i,n}|^2ds+\int_t^T\frac{3}{\varepsilon}C^2_f|\delta Y_s^{i,n}|^2ds+\int_t^T 4\varepsilon \left\{|\delta Z^{i,n}_s|^2+|\delta U^{i,n}_s|^2_\nu\right\}ds \\ \qquad\qquad\qquad\qquad\qquad
    +\int_t^T4\varepsilon\mathcal{W}^2_p(L_{n}[\textbf{Y}^n_s],\mathbb{P}_{Y^{i}_s}) ds.
    \end{array}
\end{equation*}
We obtain that, for some constant $C_p>0$, independent of $n$,  we have
\begin{equation}\begin{array}{lll}\label{eq1}
    |\delta Y^{i,n}_0|^p+\left(\int_0^T |\delta Z^{i,n}_s|^2ds\right)^{\frac{p}{2}}+\left(\int_0^T\int_{\R^*}\sum_{j=1}^n|\delta U^{i,j,n}_s(e)|^2 N^j(ds,de)\right)^{\frac{p}{2}} \leq C_p |\delta \xi^{i,n}|^p  \\ \qquad +C_p \left\{\left(2C_f+\frac{6}{\varepsilon}C_f^2\right)T\right\}^{\frac{p}{2}} \underset{0\le s\le T}{\sup} |\delta Y_s^{i,n}|^p  
    +C_p \varepsilon^{\frac{p}{2}} \left(\int_0^T |\delta Z^{i,n}_s|^2 ds\right)^{\frac{p}{2}}+C_p \varepsilon^{\frac{p}{2}} \left(\int_0^T |\delta U^{i,n}_s|^2_\nu ds\right)^{\frac{p}{2}} \\ \qquad +C_p \left(\int_0^T \mathcal{W}^2_p(L_{n}[\textbf{Y}^n_s],\mathbb{P}_{Y^{i}_s})ds\right)^{\frac{p}{2}} +C_p \left\{ \left|\int_0^T \delta Y_s^{i,n} \sum_{j=1}^n\delta Z_s^{i,j,n} dB^j_s \right|^{\frac{p}{2}} 
    \right. \\  \left. \qquad + \left|\int_0^T \int_{\R^*} \delta Y_{s^-}^{i,n} \sum_{j=1}^n\delta U_s^{i,j,n}(e) \Tilde{N}^j(ds,de) \right|^{\frac{p}{2}}  +\left|\int_0^T \delta Y_{s^-}^{i,n} d(\delta K_s^{1,i,n})\right|^{\frac{p}{2}}+\left|\int_0^T \delta Y_{s^-}^{i,n} d(\delta K_s^{2,i,n})\right|^{\frac{p}{2}}\right\}.
    \end{array}
\end{equation}
By applying the Burkholder-Davis-Gundy inequality, we derive that there exist some constants $m_p>0$ and $l_p>0$ such that
\begin{align*}
C_p \mathbb{E}\left[\left|\int_0^T \delta Y_s^{i,n} \sum_{j=1}^n\delta Z_s^{i,j,n} dB^j_s \right|^{\frac{p}{2}}\right]  & \leq m_p \mathbb{E} \left[\left(\int_0^T (\delta Y_s^{i,n})^2 |\delta Z_s^{i,n}|^2 ds\right)^{\frac{p}{4}}\right] \\ & \leq \frac{m^2_p}{2} \|\delta Y_s^{i,n}\|^p_{\mathcal{S}^p}+\frac{1}{2}\mathbb{E}\left[\left(\int_0^T |\delta Z^{i,n}_s|^2ds\right)^{\frac{p}{2}}\right]
\end{align*}
and
\begin{align*}
C_p \mathbb{E}\left[\left|\int_0^T \int_{\R^*} \delta Y_{s^-}^{i,n} \sum_{j=1}^n \delta U_s^{i,j,n}(e) \Tilde{N}^j(ds,de) \right|^{\frac{p}{2}}\right] \leq l_p \mathbb{E} \left[\left(\int_0^T (\delta Y_{s^-}^{i,n})^2 \int_{\R^*}\sum_{j=1}^n(\delta U_s^{i,j,n}(e))^2 N^j(ds,de)\right)^{\frac{p}{4}}\right] \nonumber \\ \qquad\qquad \leq \frac{l^2_p}{2} \|\delta Y_s^{i,n}\|^p_{\mathcal{S}^p}+\frac{1}{2}\mathbb{E}\left[\left(\int_0^T \int_{\R^*}\sum_{j=1}^n(\delta U^{i,j,n}_s(e))^2 N^j(ds,de)\right)^{\frac{p}{2}}\right].
\end{align*}
Also recall that, for some constant $e_p>0$ we have
\begin{align*}
    \mathbb{E}\left[\left(\int_0^T \int_{\R^*}|\delta U^{i,n}_s(e)|^2\nu(de)ds\right)^{\frac{p}{2}}\right] \leq e_p \mathbb{E}\left[\left(\int_0^T \int_{\R^*} \sum_{j=1}^n |\delta U^{i,j,n}_s(e)|^2N^j(de,ds)\right)^{\frac{p}{2}}\right].
\end{align*}
Now, we take the expectation in $\eqref{eq1}$, by using the above inequalities and taking $\varepsilon>0$ small enough, we obtain
\begin{align}\label{ii1}
    &\mathbb{E}\left[\left(\int_0^T |\delta Z^{i,n}_s|^2ds\right)^{\frac{p}{2}}+\left(\int_0^T \|\delta U^{i,n}_s\|_\nu^2ds\right)^{\frac{p}{2}}\right]  \leq  C_p |\delta \xi^{i,n}|^p \nonumber \\ &+K_{C_f,\varepsilon,T,p} \|\delta Y_s^{i,n}\|^p_{\mathcal{S}^p}+C_p \mathbb{E}\left[\sup_{0 \leq s \leq T} \mathcal{W}^p_p(L_{n}[\textbf{Y}^n_s],\mathbb{P}_{Y^{i}_s})\right]\nonumber \\ &+\mathbb{E}\left[\left(\sup_{0 \leq s \leq T}|\delta Y_s^{i,n}|(K_T^{1,i,n}+K_T^{1,i})\right)^{\frac{p}{2}}\right]+\mathbb{E}\left[\left(\sup_{0 \leq s \leq T}|\delta Y_s^{i,n}|(K_T^{2,i,n}+K_T^{2,i})\right)^{\frac{p}{2}}\right].
\end{align}
From Step 1, we have $\|\delta Y^{i,n}\|^p_{\mathcal{S}^p} \to 0$, which also implies the uniform boundedness of the sequence $\left(\|Y^{i,n}\|^p_{\mathcal{S}^p}\right)_{n \geq 0}$. Furthermore, by Assumption \ref{Assump:chaos} and Proposition \ref{KK}, we obtain that $\mathbb{E}[(K_T^{1,i,n})^p]$ (resp. $\mathbb{E}[(K_T^{2,i,n})^p]$) are uniformly bounded. Taking the limit with respect to $n$ in \eqref{ii1}, and using Theorem \ref{LLN-2}, we get the convergence $\underset{n\to\infty}{\lim}\|Z^{i,n}-Z^i{\bf e}_i\|_{\mathcal{H}^{p,n}}=0,    \,\underset{n\to\infty}{\lim}\|U^{i,n}-U^i{\bf e}_i\|_{\mathcal{H}^{p,n}_\nu}=0$.

From the equations satisfied by $W^{i,n}$ and $W^{i}$, that is
\begin{align}\label{eq2}
    W_T^{i,n}=Y_0^{i,n}-\xi^{i,n}& -\int_0^Tf(s,Y_s^{i,n}, Z_s^{i,i,n}, U_s^{i,i,n}, L_{n}[\textbf{Y}_s^n])ds \nonumber \\ & +\sum_{j=1}^n\int_0^TZ_s^{i,j,n}dB^j_s+\int_0^T\int_{\R^*} \sum_{j=1}^nU_s^{i,j,n}(e)\Tilde{N}^j(ds,de)
\end{align}
and
\begin{align}\label{eq3}
W_T^{i}=Y_0^{i}-\xi^{i}-\int_0^Tf(s,Y_s^{i}, Z_s^{i}, U_s^{i}, \mathbb{P}_{Y_s^{i}})ds+\int_0^TZ_s^{i}dB^i_s+\int_0^T\int_{\R^*} U_s^{i}(e)\Tilde{N}^i(ds,de),
\end{align}
and the convergence of $(Y^{i,n},Z^{i,n}, U^{i,n})$ shown above, we derive that $\underset{n\to\infty}{\lim}\|W^{i,n}-W^i\|_{\mathcal{S}^{p}}=0$.\qed
\end{proof}

\medskip

\begin{Corollary}[Propagation of chaos]\label{prop-y-u-z} 
Under the assumptions of Proposition \ref{chaos-1}, the particle system \eqref{BSDEParticle} satisfies the propagation of chaos property, i.e. for any fixed positive integer $k$, 
$$
\underset{n\to\infty}{\lim} \text{Law}\,(\Theta^{1,n},\Theta^{2,n},\ldots,\Theta^{k,n})=\text{Law}\,(\Theta^1,\Theta^2,\ldots,\Theta^k),
$$
where
$$
\Theta^{i,n}:=(Y^{i,n},Z^{i,n},U^{i,n},K^{1,i,n}-K^{2,i,n}),\quad \Theta^{i}:=(Y^i,Z^i,U^i,K^{1,i}-K^{2,i}).
$$
\end{Corollary}
\begin{proof}
We obtain the propagation of chaos if we can show that 
$\underset{n\to\infty}{\lim}D_{G}^p(\P^{k,n},\P_{\Theta}^{\otimes k})=0$. But, this follows from the inequality \eqref{chaos-0} and Proposition \ref{chaos-1}. \qed
\end{proof}
\begin{Remark}
The question of convergence of $S$-saddle points of the particle system to those of the limit process is more elaborate and will be addressed in a forthcoming paper. 
\end{Remark}


\appendix 
\section{Some technical results} 

For the reader's convenience, we recall here the following $L^p$ estimates with \textit{universal constants} for the difference of the solutions $(Y^{i}, Z^{i}, U^{i})$, $i=1,2$ of BSDEs, established in \cite{ddz21}.

\begin{Proposition}[$L^p$ a priori estimates with \textit{universal constants}]\label{estimates} Let $T>0$. Let $\tau$ be a $\mathbb{F}$-stopping time with values in $[0,T]$. Let $p\ge 2$ and let $\xi_1$ and $\xi_2$ $\in L^p(\mathcal{F}_\tau)$. Let $f_1$ be a Lipschitz driver with constant $C$ and let $f_2$ be a driver. For $i=1,2$, let $(Y^{i},Z^{i}, U^{i})$ be a solution of the BSDE associated to terminal time $\tau$, driver $f^{i}$, and terminal condition $\xi^{i}$. For $s \in [0,\tau]$ denote $\delta Y_s:= Y_s^{1}-Y_s^{2}$, $\delta Z_s:= Z_s^{1}-Z_s^{2}$, $\delta U_s:= U_s^{1}-U_s^{2}$, $\delta f_s:=f^{1}(s,Y_s^{2}, Z_s^{2}, U_s^{2})-f^{2}(s,Y_s^{2}, Z_s^{2}, U_s^{2})$ and $\delta \xi:=\xi^1-\xi^{2}$. Let $\eta, \beta>0$ be such that $\beta \geq 2C+\frac{3}{\eta}$ and $\eta \leq \frac{1}{ C^2}$, then for each $t \in [0,\tau]$ we have
\begin{align}
    |e^{\beta t} \delta Y_t|^{p} \leq  2^{p/2-1}\left(\mathbb{E}\left[|e^{\beta \tau} \delta \xi|^p|\mathcal{F}_t\right]+\eta^p \mathbb{E}\left[\left(\int_t^{\tau} |e^{\beta s} \delta f_s|^2 ds \right)^{p/2}|\mathcal{F}_t\right]\right) \,\,\, \P \text{-a.s.}
\end{align}
\end{Proposition}
\begin{proof}
Recall that, by standard estimates obtained by applying It\^o's formula to $e^{\beta t}|\delta Y_t|^2$, we derive that for $\beta \geq \frac{3}{\eta}+2C$ and $\eta \leq \frac{1}{C^2}$ (see \cite{dqs15}), we have
\begin{align}
    |e^{\beta t} \delta Y_t|^{2} \leq \mathbb{E}\left[|e^{\beta \tau} \delta \xi|^2|\mathcal{F}_t\right]+\eta^2 \mathbb{E}\left[\left(\int_t^{\tau} |e^{\beta s} \delta f_s|^2 ds \right)|\mathcal{F}_t\right] \,\,\, \P \text{-a.s.},
\end{align}
from which it follows that 
\begin{align}
    |e^{\beta t} \delta Y_t|^{p} \leq \left(\mathbb{E}\left[|e^{\beta \tau} \delta \xi|^2|\mathcal{F}_t\right]+\eta^2 \mathbb{E}\left[\left(\int_t^{\tau} |e^{\beta s} \delta f_s|^2 ds \right)|\mathcal{F}_t\right]\right)^{p/2} \,\,\, \P \text{-a.s.},
\end{align}
which leads to, by convexity relation and H{\"o}lder inequality,
\begin{align}
    |e^{\beta t} \delta Y_t|^{p} \leq 2^{p/2-1}\left(\mathbb{E}\left[|e^{\beta \tau} \delta \xi|^p|\mathcal{F}_t\right]+\eta^p\mathbb{E}\left[\left(\int_t^{\tau} |e^{\beta s} \delta f_s|^2 ds \right)^{p/2}|\mathcal{F}_t\right]\right) \,\,\, \P \text{-a.s.}
\end{align}
\end{proof}

We now provide the following estimates on the solution of a doubly reflected BSDEs.

\begin{Proposition}\label{KK}
Let $p \geq 2$. Let $h_1$ and $h_2$ two right-continuous left limited processes in $\mathcal{S}^p$ such that $h_1(t) \leq h_2(t)$ $\mathbb{P}$-a.s. for all $t \in [0,T]$ and $h_1$ and $h_2$ satisfy Mokobodzki's condition, i.e. there exist two non-negative supermartingales $\theta_1, \theta_2$ in $\mathcal{S}^p$  such that
\begin{equation}\label{M}
h_1(t) \le \theta_1(t)-\theta_2(t)\le h_2(t),\quad 0\le t\le T,\,\,\,  \P\text{-a.s.}
\end{equation}
Let $f$ be a Lipschitz driver, $\xi \in L^p(\mathcal{F}_T)$ and $(Y,Z,U,K^1,K^2) \in \mathcal{S}^p \times \mathcal{H}^p \times \mathcal{H}_\nu^p \times \mathcal{S}_{i}^{p} \times \mathcal{S}_{i}^{p}$ the unique solution of the doubly reflected BSDE associated with terminal time $T$, terminal value $\xi$, driver $f$ and obstacles $h_1$ and $h_2$. Then, we have
\begin{align}
    &\mathbb{E}\left[\left(\int_0^T |Z_s|^2 ds \right)^{p/2}+\left(\int_0^T \int_{\mathbb{R}^*} |U_s|_\nu^2ds \right)^{p/2}+(K_T^1)^{p} +(K_T^2)^{p}\right] \nonumber \\ & \leq C \left(\mathbb{E}[|\xi|^p]+\mathbb{E}\left[\left(\int_0^T |f(s,0,0,0)|^2ds\right)^{p/2}\right]+\|Y\|^p_{\mathcal{S}^p}+\|\theta_1\|^p_{\mathcal{S}^p}+\|\theta_2\|^p_{\mathcal{S}^p}+\|h_1\|^p_{\mathcal{S}^p}+\|h_2\|^p_{\mathcal{S}^p}\right),
\end{align}
for some constant $C$ which only depends on $p$ and $T$.
\end{Proposition}
\begin{proof}
\underline{\textit{Step 1.}} We first provide bounds on the increasing processes $K^1$ and $K^2$. More precisely, we show that
\begin{equation}\label{K}
\E[(K^1_T)^p]\le p^p \|\Theta_1\|^p_{\mathcal{S}^p},\quad  \E[(K^2_T)^p]\le p^p \|\Theta_2\|^p_{\mathcal{S}^p},
\end{equation}
where
$$
\begin{array}{lll}
\Theta_1(t):=\left(\theta_1(t)+\E[\xi^-| \mathcal{F}_t]\right)\ind_{\{t<T\}}+\E[\int_t^T f^-(s)ds\,|\,\mathcal{F}_t],\\
\Theta_2(t):=\left(\theta_2(t)+\E[\xi^+| \mathcal{F}_t]\right)\ind_{\{t<T\}}+\E[\int_t^T f^+(s)ds\,|\,\mathcal{F}_t], \\
f(t):=f(t,Y_t,Z_t,U_t),
\end{array}
$$
where $x^+:=\max(x,0)$ and $x^-=x^+-x$.

\medskip

Consider the processes defined by
$$
\begin{array}{lll}
H_1(t):=h_1(t)\ind_{\{t<T\}}+\xi\ind_{\{t=T\}}-\E[\xi+\int_t^T f(s)ds\,|\,\mathcal{F}_t],\\
H_2(t):=h_2(t)\ind_{\{t<T\}}+\xi\ind_{\{t=T\}}-\E[\xi+\int_t^T f(s)ds\,|\,\mathcal{F}_t].
\end{array}
$$
Then, $\Theta_1$ and $\Theta_2$ are non-negative supermartingales in $\mathcal{S}^p$ satisfying $\Theta_1(T)=\Theta_2(T)=0$. 

\noindent Moreover,
\begin{equation}\label{Theta}
H_1(t) \le \Theta_1(t)-\Theta_2(t)\le H_2(t),\quad 0\le t\le T,\quad \P\text{-a.s.} 
\end{equation}
Now, consider the sequence $(Y^+_n,Y^-_n)$ of processes  defined recursively as Snell envelops of processes as follows

$$\begin{array}{lll}
Y^+_{n+1}(t)=\underset{\tau\in\mathcal{T}_t}{\esssup\,}\E[Y^-_n(\tau)+H_1(\tau)|\mathcal{F}_t],\quad Y^+_0(\cdot)=0,\\ Y^-_{n+1}(t)=\underset{\tau\in\mathcal{T}_t}{\esssup\,}\E[Y^+_n(\tau)-H_2(\tau)|\mathcal{F}_t],\quad Y^-_0(\cdot)=0.
\end{array}
$$
In view of \eqref{Theta} and the properties of the Snell envelope of processes, it is easily checked that 
\begin{equation}\label{mono}
0\le Y^+_n(t)\le Y^+_{n+1}(t)\le \Theta_1(t),\quad 0\le Y^-_n(t)\le Y^-_{n+1}(t)\le \Theta_2(t),\quad 0\le t\le T,\quad \P\text{-a.s.}
\end{equation}
Therefore, the sequence $(Y^+_n)_n$ (resp. $(Y^-_n)_n$) converges pointwisely to a non-negative supermartingale $Y_1\in\mathcal{S}^p$ (resp.  $Y_2\in\mathcal{S}^p$). Moreover, $Y_1$ and $Y_2$ satisfy
\begin{equation}\label{mono}
0\le Y_1(t)\le \Theta_1(t),\quad 0\le  Y_2(t)\le \Theta_2(t),\quad 0\le t\le T,\quad \P\text{-a.s.},
\end{equation}
and 
\begin{equation}\label{Y-1-2}
\begin{array}{lll}
Y_1(t)=\underset{\tau\in\mathcal{T}_t}{\esssup\,}\E[Y_2(\tau)+H_1(\tau)|\mathcal{F}_t],\quad Y_1(T)=0,\\ Y_2(t)=\underset{\tau\in\mathcal{T}_t}{\esssup\,}\E[Y_1(\tau)-H_2(\tau)|\mathcal{F}_t],\quad Y_2(T)=0.
\end{array}
\end{equation}
Hence,
\begin{equation}\label{Y-H}
H_1(t)\le Y_1(t)-Y_2(t)\le H_2(t),\quad 0\le t\le T,\quad \P\text{-a.s.}
\end{equation}
By the Doob-Meyer decomposition, we have
$$Y_1(t)=M_1(t)-K_1(t),\quad Y_2(t)=M_2(t)-K_2(t),
$$
where $M_1$ and $M_2$ are c{\`a}dl{\`a}g martingales and $K_1, K_2$ are non-decreasing processes such that $K_1(0)=K_2(0)=0$. Furthermore, by a classical inequality for non-negative supermartingales (see e.g. Inequality (100.3) in \cite[p. 183]{DM82}), it holds that
\begin{equation}\label{K-2-2}
\begin{array}{lll}
\E[(K_1(T))^p]\le p^p \|Y_1\|^p_{\mathcal{S}^p}\le p^p \,\underset{n\to\infty}{\liminf}\,\|Y^+_n\|^p_{\mathcal{S}^p}\le p^p \|\Theta_1\|^p_{\mathcal{S}^p},\\
\\E[(K_2(T))^p]\le p^p \|Y_2\|^p_{\mathcal{S}^p}\le p^p \,\underset{n\to\infty}{\liminf}\,\|Y^-_n\|^p_{\mathcal{S}^p}\le p^p \|\Theta_2\|^p_{\mathcal{S}^p},
\end{array}
\end{equation}
which entails that $M_1$ and $M_1$ belong to $\mathcal{S}^p$. Thus, by the Martingale representation theorem, there exist unique processes $Z_1,Z_2$ in $\mathcal{H}^{p,d}$ and $U_1,U_2$ in $\mathcal{H}_{\nu}^p$ such that
$$
M_i(t)=Y_{i}(0)+\int_0^tZ_i(s)dB_s+\int_0^t\int_{R^*}U_i(s,e)\widetilde{N}(ds,de),\quad 0\le t\le T,\quad i=1,2.
$$
In view of \eqref{Y-1-2}, the by now standard arguments for Snell envelops yield that $Y_1-Y_2$ satisfies the following Skorohod flatness condition 
$$
\int_0^T (Y_1(t^-)-Y_2(t^-)-H_1(t^-))dK_1(t)=0,\quad \int_0^T (Y_1(t^-)-Y_2(t^-)-H_2(t^-))dK_2(t)=0.
$$
On the other hand, by the Martingale representation theorem applied to the martingale $\E[\xi+\int_0^T f(s)ds|\mathcal{F}_t]$, there exist unique processes $Z_3$ in $\mathcal{H}^{p,d}$ and $U_3$ in $\mathcal{H}_{\nu}^p$ such that
$$
\E[\xi+\int_t^T f(s)ds|\mathcal{F}_t]=\E[\xi+\int_0^T f(s)ds]+\int_0^tZ_3(s)dB_s+\int_0^t\int_{R^*}U_3(s,e)\widetilde{N}(ds,de)-\int_0^t f(s)ds.
$$
Set
$$
\widehat{Y}(t):=Y_1(t)-Y_2(t)+\E[\xi+\int_t^T f(s)ds|\mathcal{F}_t],\,\, \widehat{Z}(t):=Z_1(t)-Z_2(t)+Z_3(t),\,\, \widehat{U}(t):=U_1(t)-U_2(t)+U_3(t).
$$
Then, it is easy to check that the process $(\widehat{Y}, \widehat{Z},\widehat{U}, K_1,K_2)$  is a solution to the  doubly reflected BSDE associated with $(f(t), h_1(t),h_2(t))$. By uniqueness of the solution to \eqref{BSDE1}, we must have
$$
(\widehat{Y}, \widehat{Z},\widehat{U}, K_1,K_2)=(Y,Z,U,K^1,K^2).
$$ 
In particular, $K^1$ and $K^2$ satisfy the inequalities \eqref{K-2-2}. 

\underline{\textit{Step 2.}} We will now give estimates on the whole solution $(Y,Z,U,K^1,K^2)$. By applying It\^o's formula on $|Y_t|^2$, we get
\begin{align*}
    &|Y_t|^2+\int_t^T |Z_s|^2ds+\int_t^T\int_{\R^*}\sum_{j=1}^n|U_s(e)|^2 N^j(ds,de)+\sum_{t<s\leq T}|\Delta K_s^{1}-\Delta K_s^{2}|^2=  |\xi|^2 \nonumber \\
    &  + 2\int_t^T Y_s f(s) ds - 2 \int_t^T Y_s \sum_{j=1}^n Z_s dB^j_s - 2 \int_t^T \int_{\R^*} Y_{s^-} \sum_{j=1}^n U_s(e) \Tilde{N}^j(ds,de) \nonumber \\ & +2\int_t^T Y_{s^-} dK_s^{1}-2\int_t^T Y_{s^-} dK_s^{2} \nonumber \\
    &= |\xi|^2 + 2\int_t^T Y_s f(s) ds - 2 \int_t^T Y_s \sum_{j=1}^n Z_s dB^j_s - 2 \int_t^T \int_{\R^*} Y_{s^-} \sum_{j=1}^n U_s(e) \Tilde{N}^j(ds,de) \nonumber \\ & +2\int_t^T h_1(s^-) dK_s^{1}-2\int_t^T h_2(s^-) dK_s^{2}.
\end{align*}
By the Lipschitz property of $f$ and Young's inequality, we have, for all $\varepsilon>0$,
\begin{align}
    \int_t^T f(s) Y_sds \leq 2\int_t^T |f(s,0,0,0)|^2ds+\left(2+2C_f^2 \varepsilon^{-1}\right)\int_t^T |Y_s|^2ds+\varepsilon \int_t^T Z_s^2ds+\varepsilon \int_t^T |U_s|_\nu^2ds.
\end{align}
We  obtain that, for a constant $C_p>0$ and $0 \leq t \leq T$, the following inequality holds
\begin{align*}
    &|Y_t|^{p}+\left(\int_t^T |Z_s|^2ds\right)^{p/2}+\left(\int_t^T\int_{\R^*}\sum_{j=1}^n|U_s(e)|^2 N^j(ds,de)\right)^{p/2} \leq  C_p|\xi|^2 \nonumber \\
    &  + C_p\left(\int_t^T |f(s,0,0,0)|^2ds\right)^{p/2} + C_p\left(2+2C_f^2 \varepsilon^{-1}\right)^{p/2}\left(\int_t^T |Y_s|^2ds\right)^{p/2}+C_p\varepsilon^{p/2}\left(\int_t^T |Z_s|^2ds\right)^{p/2} \nonumber \\ & +C_p\varepsilon^{p/2}\left(\int_t^T |U_s|_\nu^2ds\right)^{p/2}+ C_p \left|\int_t^T Y_s \sum_{j=1}^n Z_s dB^j_s\right|^{p/2} + C_p \left|\int_t^T \int_{\R^*} Y_{s^-} \sum_{j=1}^n U_s(e) \Tilde{N}^j(ds,de)\right|^{p/2} \nonumber \\ & +C_p \left|\int_t^T h_1(s^-) dK_s^{1}\right|^{p/2}+C_p \left|\int_t^T h_2(s^-) dK_s^{2}\right|^{p/2}.
\end{align*}
Therefore, for any $\varepsilon>0$, and a possibly different constant $C_p$, we have
\begin{align}\label{in0}
    &\left(\int_0^T |Z_s|^2ds\right)^{p/2}+\left(\int_0^T\int_{\R^*}\sum_{j=1}^n|U_s(e)|^2 N^j(ds,de)\right)^{p/2} \leq  C_p|\xi|^2 \nonumber \\
    &  + C_p\left(\int_0^T |f(s,0,0,0)|^2ds\right)^{p/2} + T^pC_p\left(2+2C_f^2 \varepsilon^{-1}\right)^{p/2}\sup_{0 \leq s \leq T} |Y_s|^p+C_p\varepsilon^{p/2}\left(\int_0^T |Z_s|^2ds\right)^{p/2} \nonumber \\ & +C_p\varepsilon^{p/2}\left(\int_0^T |U_s|_\nu^2ds\right)^{p/2}+ C_p \left|\int_0^T Y_s \sum_{j=1}^n Z_s dB^j_s\right|^{p/2} + C_p \left|\int_0^T \int_{\R^*} Y_{s^-} \sum_{j=1}^n U_s(e) \Tilde{N}^j(ds,de)\right|^{p/2} \nonumber \\ & +C_p\varepsilon^{-p/2} \left[\left(\sup_{0 \leq t \leq T} |h^1(t)|^p+\sup_{0 \leq t \leq T} |h^2(t)|^p\right)\right]+C_p\varepsilon^{p/2} \mathbb{E}\left[  (K_T^1)^p+(K_T^2)^p\right].
\end{align}
Now, by applying the Burkholder-Davis-Gundy inequality, for some constants $d_p>0$ and $e_p>0$, we have
\begin{align}\label{d1}
C_p \mathbb{E}\left[\left|\int_0^T Y_s \sum_{j=1}^n Z_s^{j} dB^j_s \right|^{\frac{p}{2}}\right]  & \leq d_p \mathbb{E} \left[\left(\int_0^T  Y_s^2 |Z_s|^2 ds\right)^{\frac{p}{4}}\right] \\ & \leq \frac{d^2_p}{2} \| Y_s\|^p_{\mathcal{S}^p}+\frac{1}{2}\mathbb{E}\left[\left(\int_0^T |Z_s|^2ds\right)^{\frac{p}{2}}\right]
\end{align}
and
\begin{align}\label{d2}
C_p \mathbb{E}\left[\left|\int_0^T \int_{\R^*} Y_{s^-} \sum_{j=1}^n U_s^{j}(e) \Tilde{N}^j(ds,de) \right|^{\frac{p}{2}}\right] \leq e_p \mathbb{E} \left[\left(\int_0^T Y_{s^-}^2 \int_{\R^*}\sum_{j=1}^n(U_s^{j}(e))^2 N^j(ds,de)\right)^{\frac{p}{4}}\right] \nonumber \\ \qquad\qquad \leq \frac{e^2_p}{2} \| Y_s\|^p_{\mathcal{S}^p}+\frac{1}{2}\mathbb{E}\left[\left(\int_0^T \int_{\R^*}\sum_{j=1}^n(U^{j}_s(e))^2 N^j(ds,de)\right)^{\frac{p}{2}}\right].
\end{align}
Furthermore, for some constant $l_p>0$ we have:
\begin{align}\label{d3}
    \mathbb{E}\left[\left(\int_0^T \int_{\R^*}|U_s(e)|^2\nu(de)ds\right)^{\frac{p}{2}}\right] \leq l_p \mathbb{E}\left[\left(\int_0^T \int_{\R^*} \sum_{j=1}^n |U^{j}_s(e)|^2N^j(de,ds)\right)^{\frac{p}{2}}\right].
\end{align}
By taking the expectation in \eqref{in0}, and using \eqref{K}, \eqref{d1}, \eqref{d2} and \eqref{d3}, for small enough $\varepsilon$, the result follows.

\qed

\end{proof}


{}

\end{document}